\documentclass[a4paper]{amsart}

\usepackage{amssymb} 
\usepackage{graphicx} 
\usepackage{mathtools}
\usepackage{graphicx}
\usepackage{mathrsfs}
\usepackage{amsfonts}
\usepackage{enumerate}
\usepackage{latexsym, euscript, epic, eepic, color, hyperref, tikz}
\usepackage{enumitem}
\usepackage{subcaption} 
\usepackage{float} 
\usepackage{pgfplots}
\usepgfplotslibrary{fillbetween} 

\usetikzlibrary{arrows.meta}

\def\brackets#1{\left( #1 \right)}

\def\bigbrackets#1{\left\{ #1 \right\}}

\newcommand{\NN}{\mathbb{N}}
\newcommand{\RR}{\mathbb{R}}

\newcommand{\eps}{\varepsilon}
\newcommand{\mathdot}{~:~}

\newcommand{\lbd}{\underline{\dim}_{\mathrm{B}}}

\newcommand{\bd}{\dim_{\mathrm{B}}}

\newcommand{\asd}{\dim_{\mathrm{A}}}

\newcommand{\lasp}{\dim_{\mathrm{L}}^\theta}

\newcommand{\qlad}{\dim_\mathrm{qL}}

\newtheorem{lemma}{Lemma}
\newtheorem{theorem}{Theorem}

\newtheorem{proposition}{Proposition}
\newtheorem{example}{Example}
\newtheorem{corollary}{Corollary}
\newtheorem{claim}{Claim}

\numberwithin{equation}{section}

\newcommand{\lad}{\dim_{\rm L}}
\newcommand{\lqad}{\underline{\dim}_{\rm L}}

\newcommand{\lbxd}{\underline{\dim}_{\rm B}}

\title{Generalised lower Assouad-type dimensions and their interpolations}

\author{Haipeng Chen}
\address{School of Artificial Intelligence, Shenzhen Technology University, Shenzhen, China, 518118}
\email{hpchen0703@foxmail.com}
\author{Wen Wang\textsuperscript{*}}
\address{School of Mathematics and Statistics, Yunnan University, Kunming, China, 650504}
\email{sophia\_84@126.com}
\thanks{\textsuperscript{*}:Corresponding author.}

\keywords{Generalised Lower Assouad dimension, Lower Assouad spectra, Lower Assouad dimension.}
\date{\today}
\subjclass[2020]{28A80 (Primary), 28A78 (Secondary)}

\begin{document}

\maketitle

\begin{abstract}
This paper investigates the analytic and structural properties of the $\phi$-lower Assouad dimension, a generalized notion extending the lower Assouad dimension. We establish the equivalence of $\phi$-lower Assouad dimensions with respect to the dimension functions, prove analytic properties related to the regularity of the $\phi$-lower dimension, and analyse the role of rate windows in this context. Furthermore, we explore both positive and negative interpolation properties of the $\phi$-lower dimension by presenting corresponding theorems that delineate these behaviors.
\end{abstract}

\tableofcontents

\section{Introduction}

\subsection{History of lower Assouad type dimensions}

A fundamental research direction in fractal geometry and geometric measure theory is to elucidate the fine-scale structure of sets with irregular geometry. The lower Assouad dimension is an important quantity in the dimension theory literature and related fields. It was first introduced as the {\it minimal dimensional number} by Larman in 1967 \cite{L1967}. In recent decades, the lower Assouad dimension has emerged as the natural dual to the well-studied Assouad dimension and serves as the minimal Hausdorff dimension of limit sets obtained by ‘zooming in’ on compact sets. The lower Assouad dimension and its variants have received widespread attention in the dimension theory of fractal geometry. We refer the reader to \cite{JF2020,FHKY2019} for a general introduction and further details.

We work with the non-empty, bounded doubling metric space $(X,d)$, which is simplified as {\it n.b.d. space $X$} for the rest of this paper. For any $F \subseteq X$, we denote by $N_r(F)$ the smallest number of balls of radius needed to cover $F$. The lower Assouad dimension is defined as 
\begin{align*}
\dim_{\rm L} F & = \sup \{ s \geq 0 ~:~ \text{ there exist constant } C > 0 \text{ such that }\\
& \hspace{1 cm} \text{ for all } 0 < r < R \text{ and } x \in F, ~N_r(B(x,R) \cap F) \geq C \left( \frac{R}{r}\right)^s \}.
\end{align*}

It follows from the definition that the lower Assouad dimension measures the sparsest (thinnest) local scaling property, capturing minimal covering complexity within a set. For well-behaved sets, such as Ahlfors-David regular sets $F$, then $\lad F = \bd F = \asd F$ where these denote the lower Assouad, box-counting, and Assouad dimensions, respectively. However, certain irregular behaviors can arise specifically for the lower Assouad dimension: for instance, the lower Assouad dimension of any set containing isolated points is always zero. Moreover, there exist open set $U = \cup_{n=1}^\infty  (1/n- 1/2^{n+1}, 1/n + 1/2^{n+1})$ such that $\lad U = 0$. Details see Fraser \cite{JF2020}.

It is instructive to compare the lower Assouad dimension with other classical fractal dimensions. In general, the lower Assouad dimension is the smallest among standard notions of dimension. For example, for all $F \subseteq X$, 
\[
\dim_{\rm L} F \leq \underline{\dim}_{\rm B} F \leq 
\overline{\dim}_{\rm B} F\]
and for all compact sets, 
\[
\dim_{\rm L} F \leq \dim_{\rm H} F.
\]
Nevertheless, for certain non-compact sets, like $[0,1] \cap \mathbb{Q}$, the lower Assouad dimension can actually exceed the Hausdorff dimension. Understanding how such coarse scaling arises and how different notions of dimensions have played an important role in the development of fractal geometry and related fields.

\subsection{Generalised lower dimensions and the rate windows.}

In this paper, we investigate the fine scaling properties of the lower Assouad dimension, both in a general framework and for several notable explicit families of sets. A particularly prominent example in this context is the lower Assouad spectrum, which serves as the natural dual to the Assouad spectrum; both notions were introduced by Fraser and Yu in \cite{FY2018}. The lower Assouad spectrum is a variation on the lower Assouad dimension, obtained by restricting the scaling relationship to $r =  R^{1/\theta}$ for some fixed $\theta \in (0,1)$ in the lower Assouad dimension. This notion has also contributed to the broader program of dimension interpolation. For more information on the dimension interpolation, see Fraser \cite{JF2020} and references therein.

We denote the lower Assouad spectrum by $\lasp F$, and the quasi-lower Assouad spectrum by $\underline{\dim}_{\rm L}^\theta F$ repsectively. It is a continuous parameterised family of dimensions with 
\[
\lad F \leq \underline{\dim}_{\rm L}^\theta F \leq \dim_{\rm L}^\theta F \leq \lbxd F.
\]
The robust theory of the lower Assouad spectrum has attracted considerable attention in recent years due to its unique scaling behaviors. Notably, the lower Assouad spectrum exhibits different phenomena from the classical Assouad spectrum. While the Assouad spectrum interpolates between the box dimension and the quasi-Assouad dimension, the lower Assouad spectrum does not necessarily follow this pattern. Instead, as shown by Chen et al.\cite{CDW2017, CWC2020}, both $\lim\limits_{\theta \to 0} \dim_{\rm L}^\theta F$ and $\lim\limits_{\theta \to 1} \dim_{\rm L}^\theta F$ exist, and the lower Assouad spectrum interpolates between the {\it quasi-lower Assouad dimension} and the lower box dimension with 
\[
\dim_{\rm L} F \leq \dim_{\rm qL} F = \lim\limits_{\theta \to 1} \dim_{\rm L}^\theta F \leq \lim\limits_{\theta \to 0} \dim_{\rm L}^\theta F \leq \underline{\dim}_{\rm B} F,
\]
but $\dim_{\rm qL} F$ may strictly larger than $\dim_{\rm L} F$, and $\lim\limits_{\theta \to 0} \dim_{\rm L}^\theta F$ may strictly smaller than the $\underline{\dim}_{\rm B} F$. For certain fractal sets, such as McMullen sets, it is known that $\lim\limits_{\theta \to 0} \dim_{\rm L}^\theta F$ coincides with the lower box dimension, see \cite{JF2020}. Nevertheless, Chen et al. \cite{CWC2020} demonstrate that this coincidence does not hold in general. The lower Assouad-type dimensions are dedicated to capture the coarsest scaling properties of sets, but the lower Assouad spectrum fails to fully interpolate from lower Assouad dimension to lower box dimension in general.

Responding to this challenge, G\'arcia et al.\cite{GHM2021} introduced variations of the Assouad and lower Assouad dimensions, termed the {\it upper $\phi$-Assouad dimension} and {\it lower $\phi$-Assouad dimension}, respectively. In these definitions, the smaller scale $r$ is prescribed as a function of the larger scale 
$R$, as suggested by Fraser and Yu \cite{FY2018}. Garcia et al.\cite{GHM2021} investigated the general properties of these dimensions and established results for several appropriately chosen functions. For further developments and related works on variations of the Assouad dimension, we refer readers to \cite{GHM2021}.

In contrast to G\'arcia et al.\cite{GHM2021}., Banaji et al.\cite{BRT2023} considered a different variation of the Assouad dimension by restricting $r$ via a general function of $R$, termed the {\it $\phi$-Assouad dimension}. They demonstrated that  the upper $\phi$-Assouad dimension can be derived from $\phi$-Assouad dimension using variational principle, and fully interpolates between the upper box dimension and the Assouad dimension. For a detailed discussion of the general properties of the $\phi$-Assouad dimension and its values for specific fractal sets, see \cite{BRT2023}. For further developments and related works on variations of the Assouad dimension, we refer readers to \cite{GHM2021}.

To capture the effect of lower Assouad-like dimension under the restriction that $r$ is constrained by a general dimension function of $R$ and to further examine whether, in the sense of  functions, the lower Assouad-like dimension fully interpolates between lower Assouad dimension and lower box dimension, we introduce in this paper a variant of the lower Assouad dimension by prescribing the smaller scale $r$ as a function of the larger scale $R$. We systematically investigate the fundamental properties of this new dimension, incluing the equivalence, regularity and interpolation properties. 

Let $\phi : R \in (0,1) \to \mathbb{R}^+$ be a real function, we call $\phi(R)$ be a {\it dimension function} if it satisfies:
\begin{enumerate}
\item $\phi(R)$ decreases as $R \to 0$;
\item $R \to \phi(R) \log(1/R)$ increases to infinity as $R$ decreases to $0$.
\end{enumerate}
Throughout this paper, we focus on non-empty, bounded, doubling metric spaces, abbreviated as {\it n.b.d.\ spaces}. Given n.b.d. space $F$, we call the {\it $\phi$-lower Assouad dimension}, simplified as {\it $\phi$-lower dimension}, as follows:
\begin{align*}
\lad^\phi F = \sup\{s \geq 0 \mathdot & \exists \,  C > 0, \text{ such that }  \forall \,  0 < r = R^{1+ \phi(R)} < R < 1, \\
& \inf_{x \in F} N_r(B(x,R) \cap F)  \geq \, C  R^{-\phi(R) s} \}.
\end{align*}
The {\it quasi $\phi$-lower Assouad dimension}, simplied as {\it quasi $\phi$-lower dimension} and called lower $\phi$-Assouad dimension in \cite{GHM2021}, is defined as follows:
\begin{align*}
\underline{\dim}_{\rm L}^\phi F = \sup\{s \geq 0 \mathdot & \exists \,  C > 0, \text{ such that }  \forall \,  0 < r \leq R^{1+ \phi(R)} < R < 1, \\
& \inf_{x \in F} N_r(B(x,R) \cap F)  \geq \, C  (R/r)^{s} \}.
\end{align*}
It directly follows that 
\[
\lad F \leq \underline{\dim}_{\rm L}^\phi F \leq \lad^\phi F \leq \lbd F.
\]

Additionally, to capture the finer scaling properties of $\phi$-lower dimension, given dimension function $\phi$, for any $ \alpha \in (0,\infty)$, we define $\phi_\alpha : R \mapsto \phi(R)/\alpha$ and denote by the corresponding {\it $\phi_\alpha$-lower dimension} $\lad^{\phi_\alpha} F $. We also call $\mathcal{W}_\phi$ the {\it dimension rate windows of $\phi(R)$} as 
\[
\mathcal{W}_{\phi} = \{\phi_\alpha ~:~ \alpha \in (0,\infty) \}.
\]

Notice that for sets containing isolated points or some open sets, for any dimension function $\phi$, their $\phi$-lower dimension are always 0. To address this limitation and avoid certain pathological behaviors, we also introduce the {\it $\phi$-modified lower Assouad dimension}, simplified as {\it $\phi$-modified lower dimension}, as follows:
\[
\dim_{\rm ML}^\phi F = \sup \{ \dim_{\rm L}^\phi E ~:~ E \subseteq F \}.
\]

We summarize the $\phi$-lower dimension, $\phi$-quasi lower dimension, $\phi$-modified lower dimension as the {\it generalised lower dimensions}. The main contributions of this article are summarized as follows:

\begin{enumerate}
\item We study the properties of generalised lower dimensions.
\item We investigate the rate windows and the analytic properties of $\phi$-lower dimension.
\item We discuss the interpolation problem of the generalised lower dimensions from lower dimension and quasi-lower dimension.
\item We study some examples indicating that the generalised lower dimensions fails to fully interpolate from lower dimension to the lower box dimension.
\end{enumerate}

For the rest of this paper, we will introduce our main results more precisely along this themes.

\subsection{Main results and remarks}

This section is devoted to a detailed presentation of our main theorems and their implications. For the properties of generalised lower dimensions, we obtain:

\begin{theorem}\label{THM_equivalence}
Let $\phi(R)$ and $\psi(R)$ be two different dimension functions, then the following conditions are equivlant:
\begin{enumerate}
\item $\lim\limits_{R \to 0} \phi(R)/\psi(R) = 1$.
\item For all bounded $F \subseteq \mathbb{R}$, $\lad^\phi F = \lad^\psi F$.
\item For all n.b.d spaces $F$, $\lad^\phi F = \lad^\psi F$.
\end{enumerate}
\end{theorem}

We specialize the rate windows and the analytic properties of $\phi$-lower dimension as follows.

\begin{theorem}\label{THM_variational_principle}
Let $F$ be a n.b.d. space and let $\phi(R)$ be a dimension function, then 
\[
\underline{\dim}_{\rm L}^\phi F  =  \inf_{0 < \alpha < 1} \dim_{\rm L}^{\phi_\alpha} F 
\] 
\end{theorem}

For the interpolation properties of $\phi$-lower dimension from lower Assouad dimension to quasi-lower Assouad dimension, we will show

\begin{theorem}\label{THM_interpolation_positive}
Given $F \subseteq \mathbb{R}^d$ and $s \in [\dim_{\rm L} F, \dim_{\rm qL} F]$, then there exist a dimension function $\phi: R \in (0,1) \to [0,d]$ such that 
\[
\dim_{\rm L}^\phi F = \underline{\dim}_{\rm L}^\phi F  = s.
\]
\end{theorem}

Notably, the lower Assouad dimension admits a more general recovery as follows.
\begin{theorem}\label{THM_lower_dim}
Given $F \subseteq \mathbb{R}^d$. For any dimesion function $\phi(R)$, we can choose a dimension funtion $\Psi(R)$ satisfying $\Psi(R) \leq \phi(R)$ for all sufficiently small $R  > 0$ such that $\dim_{\mathrm{L}}^{\Psi}F = \underline{\dim}_{\mathrm{L}}^{\Psi}F = \dim_{\mathrm{L}}F$.
\end{theorem}

However, the $\phi$-lower dimensions and modified $\phi$-lower spectrum fail to fully interpolate between the lower Assouad dimension to the lower box dimension as follows, which means there exist examples such that neither the $\phi$-lower dimension, quasi $\phi$-lower dimension and the modified $\phi$-lower dimension satisfy the interpolation property between $(\qlad F, \lbd F]$.
\begin{theorem}\label{THM_interpolation_negative}
There exist a compact set $F \subseteq \RR^2$ such that for all dimension functions $\phi(R)$, such that $\dim_{\rm L}^\phi F = \underline{\dim}_{\rm L}^\phi F < \dim_{\rm ML}^\phi F < \underline{\dim}_{\rm B} F$.
\end{theorem}

\subsection{Organisation of this paper}

For the rest of this paper, we provide the necessary preliminaries in Section 2, and investigate the  properties of the generalised lower dimensions in Section 3. Specifically, in Section 3.1, we discuss the equivalence of certain notions and prove Theorem \ref{THM_equivalence}. In Section 3.2, we study the class of rate windows determined by $\phi(R)$ and examine the analytic properties of the $\phi$-lower dimension, culminating in Theorem \ref{THM_variational_principle}. Section 3.3 is devoted to the positive interpolation properties of the $\phi$-lower dimension, with the main result stated in Theorem \ref{THM_interpolation_positive} and \ref{THM_lower_dim}, while Section 3.4 focuses on the negative interpolation aspect, as encapsulated in Theorem \ref{THM_interpolation_negative}.

\section{Preliminaries}

\subsection{Notations and supporting lemmas}

In this part, we introduce some notations and auxillary lemmas. Let $a , b \in \mathbb{R}$. We write $a \lesssim b$ if there exist a constant $c$ such that $a \leq cb$, and $a \gtrsim b$ if there exist a constant $c$ such that $a \geq cb$. We also write $a \approx b$ if $a \lesssim b$ and $a \gtrsim b$. 

Given a n.b.d. space $F$, for any $r > 0$, we define $N_r(F)$ to be the mininal number of balls of radius $r$ needed to cover $F$, and $M_r(F)$ to be the maximal number of discrete subset of $F$ with distance $r$. Given n.b.d. space $F$, it follows from doubling property that there exist a constant $C  > 1$ such that for any $x \in F$ and $r > 0$, $N_{r}(B(x,2r) \cap F) \leq C$, and then
\begin{equation}\label{doubling property}
 N_r(F) \approx M_r(F),
\end{equation}
this indicates $N_r(F)$ can be controlled by $M_r(F)$ and conversely, up to multiplicative constants. More details are recommended to refer \cite{KF2014, JF2020} and references therein. 

For any $\theta \in (0,+\infty)$, we write $R_\theta =  R^{1 + \phi(R)/\theta}$. Based on the doubling property, the following result gives an upper bound estimate for $N_{R_\alpha}(B(x,R) \cap F)$ in terms of $N_{R_\beta}(B(x,R) \cap F)$ for all $0 < \alpha < \beta < \infty$.
\begin{lemma}\label{upper bound estimate}
For n.b.d. space $F$, for any sufficiently small $R > 0$,
\[
N_{R_\alpha}(B(x,R) \cap F) \lesssim N_{R_\beta}(B(x,R) \cap F)  \cdot \sup_{y \in F} N_{R_\alpha}(B(y,R_\beta) \cap F)
\]
\end{lemma}

For the finer lower bound estimates of cardinality of ball-covers or discrete subsets, the following result, from Hare and Troscheit\cite{HT2019}, gives a more refined lower bound for $M_r(F)$.

\begin{lemma}[\cite{HT2019}]\label{lower bound estimate}
Given $F \subseteq \RR^d$, $R > 0$ and $0 < r := r_k < r_{k-1} < \dots < r_1 < R$, then for any $x \in F$,
\begin{align*}
M_r(B(x,R) \cap F) & \geq M_{r_1}(B(x,R-r_1) \cap F) \cdot \inf_{y_1 \in F} M_{r_2}(B(y_1,r_1 - r_2) \cap F) \cdot \cdots \\
& \qquad \cdot \inf_{y_{k-1} \in F} M_{r_k}(B(y_{k-1}, r_{k-1}-{r_k}) \cap F).
\end{align*}
\end{lemma}

We will repeatedly use the following elementary but crucial lemmas for generalised lower Assouad type dimensions. The first lemma introduce the doubling property of $R^{-\phi(R)}$ where $\phi(R)$ is a dimension function.
\begin{lemma}\label{lem3}
Let	$\phi$ be a dimension function, then for any positive constant $0 < C< 1$, there exist a constant $M$ related to $\phi$ such that for any sufficiently small $R > 0$, we have
\begin{equation}\label{lem3 eq1}
R^{-\phi (R)} < 	(CR)^{-\phi(CR)}<C^{-M} R^{-\phi (R)}.
\end{equation}
\end{lemma}
\begin{proof}

It follows from the definition of dimension function that there exist a constant $M  > 0 $ such that for any sufficiently small $R > 0$, $\phi(R) < M $, $\phi(R)$ decreases to $0$ and $R^{-\phi(R)}$ increase to infinity as $R \to 0$. Hence
\begin{equation}\label{lem3 eq2}
(CR)^{-\phi(CR)}\leq (CR)^{-\phi(R)} \leq C^{-M}  R^{-\phi (R)}.
\end{equation}
\end{proof}

Based on Lemma \ref{lem3}, we have the following lemma, which is important to give the lower bound of Lemma \ref{lower bound estimate} on certain scales.

\begin{lemma}\label{lem4}
Given n.b.d. space $F$ and dimension function $\phi$, let $s  = \dim_{\rm L}^{\phi} F > 0$, then for any $ \varepsilon>0$ and any $0 < C<1$, there exist $R'>0$ such that for any $0 < R < R'$, we have
	\begin{equation}\label{lem4 eq1}
		M_{R^{1+\phi(R)}}(B(x,CR)\cap F)\geq R^{-\phi(R)(s-\varepsilon)}
	\end{equation}
\end{lemma}
\begin{proof}
Given $ 0 < C <1$, it follows from the doubling property that for any $x \in F$ and any $R > 0$, we obtain
\begin{align}\label{lem4 eq2}
\nonumber M_{(CR)^{1+\phi(CR)}}& (B(x,CR)\cap F)\\ 
		\leq&M_{R^{1+\phi(R)}}(B(x,CR)\cap F) \cdot \sup_{x' \in F} N_{(CR)^{1+\phi(CR)}}(B(x',R^{1+\phi(R)})\cap F).
\end{align}
According to Lemma \ref{lem3}, 
\[
R^{-\phi (R)} < 	(CR)^{-\phi(CR)}< C^{-M} R^{-\phi (R)}.
\]
Hence, by douling property of $F$, there exist a constant $C_1 > 0$ such that 
\begin{equation}\label{lem4 eq3}
\sup_{x' \in F} N_{(CR)^{1+\phi(CR)}}(B(x',R^{1+\phi(R)})\cap F)\leq C_1.
\end{equation}
By (\ref{lem4 eq2}), (\ref{lem4 eq3}), we have
\begin{align}\label{ll}
	M_{R^{1+\phi(R)}}(B(x,CR)\cap F)\geq C_1^{-1} M_{(CR)^{1+\phi(CR)}}(B(x,CR)\cap F).
\end{align}
For any $\varepsilon>0$, by $s = \dim_{\rm L}^{\phi} F $ and (\ref{lem4 eq3}), then for any sufficiently small $R > 0$, we have $M_{R^{1+\phi(R)}}(B(x,R)\cap F)\geq R^{-\phi(R)(s-\varepsilon/2)}$ and $C_1^{-1}>R^{\varepsilon\phi(R)/2}$. As a consequence, by (\ref{ll}) and monotonocity of $R^{-\phi(R)}$, it gives
\begin{align*}
M_{R^{1+\phi(R)}}(B(x,CR)\cap F)  & \geq C_1^{-1} \cdot (CR)^{-\phi(CR)(s-\varepsilon/2)} \\
& \geq R^{\varepsilon\phi(R)/2} \cdot R^{-\phi(R)(s-\varepsilon/2)} \geq R^{-\phi(R)(s-\varepsilon)},
\end{align*}
which gives the desired argument.
\end{proof}

\subsection{Dimensions of Moran constructions}

In this part, we study the dimension formula of Moran constructions. Dimension theory of Moran constructions are essential for the discussion on the properties of dimensions, and widely applied to the study on fractal geometry.

We first recall the definition of homogeneous Moran sets in $\RR^d$. Let $\mathcal{I} = \{0,1\}^d$. We denote $\mathcal{I}^* = \cup_{n=0}^\infty \mathcal{I}^n$, and denote $\mathcal{I}^0$ by the unique word $\{ \emptyset \}$ with length $0$. Given a sequence for contraction ratios $\mathbf{r} = \{r_n \}_{n=1}^\infty$ satisfy $0 < r_n \leq 1/2$ for all $n \in \NN$. For any $\mathbf{i} \in \mathcal{I}$, we write 
\[
S_{\mathbf{i}}^n(x) = r_n x + \lambda_{\mathbf{i}}^n
\]
where the $\lambda_{\mathbf{i}}^n  = ((\lambda_{\mathbf{i}}^n )^{(1)}, (\lambda_{\mathbf{i}}^n )^{(2)}, \dots , (\lambda_{\mathbf{i}}^n )^{(d)})\in \RR^d$ is
\[
(\lambda_{\mathbf{i}}^n )^{(j)} = 
\begin{cases}
0 & \text{ if } \mathbf{i}^{(j)} = 0  ;\\
1 - r_n & \text{ if } \mathbf{i}^{(j)} = 1 .
\end{cases} \qquad \forall 1 \leq j \leq d.
\]
For any word $\sigma = (\mathbf{i}_1, \cdots ,\mathbf{i}_n )\in \mathcal{I}^*$, we extend the contraction map as $S_{\sigma} = S_{\mathbf{i}_1} \circ \cdots \circ S_{\mathbf{i}_n}$. Thus we set
\[
M_n = \bigcup\limits_{\sigma \in \mathcal{I}^n}S_{\sigma}([0,1]^d) \qquad \text{ and } \qquad M:= M(\mathbf{r}) = \bigcap\limits_{n=1}^\infty M_n
\]
and refer $M$ to the homogeneous Moran set determined by $\mathbf{r}$.

We first investigate the explicit formulas for homogeneous Moran sets. Given $\mathbf{r} = \{r_n \}_{n=1}^\infty$, for any $n \geq 1$, we denote by $\rho(n)$ the side length of $n$-th cylinder as
\[
\rho(n) = r_1 r_2 \dots r_n. 
\]
It follows from the definition of $\phi$-lower dimension that for all n.b.d. space $F$,
\begin{equation}\label{phi dim formula}
\dim_{\rm L}^\phi F  = \liminf\limits_{ R \to 0 } \dfrac{\inf\limits_{x \in F} \log N_{R^{1+\phi(R)}}(B(x,R) \cap F)}{- \phi(R)\log R},
\end{equation}
by which we can get the explicit formula of $\phi$-lower Assouad dimension of homogeneous Moran sets.

We now discuss the dimension formula of homogeneous Moran sets. We first give some terms and notations. Given $R > 0$, we denote $l(R)$ be the largest integer satisfying 
\begin{equation}\label{l(R)}
\rho(l(R) + 1) < R \leq \rho(l(R))
\end{equation}
and $l_\phi(R)$ be the largest integer satisfying 
\begin{equation}\label{l_phi(R)}
\rho(l(R) + l_\phi(R) + 1) < R^{1+\phi(R)} \leq \rho(l(R)+ l_\phi(R)).
\end{equation}
The following result presents an explicit formula for the generalized lower Assouad dimension of homogeneous Moran constructions.

\begin{proposition}\label{homo_moran_formula}
Suppose $r_* : = \inf\limits_{n \geq 1} r_n  > 0$, then for any homogeneous Moran set $M \subseteq \mathbb{R}$, it gives
\[
\dim_{\rm L}^\phi M = \liminf_{n \to \infty} \dfrac{l_\phi(\rho(n)) \log 2}{\log \rho(n)/\rho(n + l_\phi(\rho(n)))}.
\]
\end{proposition}

\begin{proof}

 This is a direct result from the definition, but we state the proof here for completeness. For any $x\in M$ and sufficiently small $R >0 $, the ball $B(x,R)$ intersects at most two cylinders of level $l(R)$, and contains at least one cylinder of $l(R)+1$. For any level $k \in \mathbb{N}^+$, we denote $I_k$ by the $k$-th cylinder of $M$. Thus, by (\ref{l(R)}) and (\ref{l_phi(R)}), it gives
\begin{equation}\label{p1 eqn1}
N_{\rho(l(R)+l_{\phi}(R))}(I_{l(R)+1}) \leq \inf_{x\in M} 
 N_{R^{1+\phi(R)}}(B(x,R)\cap M) \leq 2 N_{\rho(l(R)+l_{\phi}(R)+1)}(I_{l(R)})
\end{equation}
and
\begin{equation}\label{p1 eqn2}
\frac{\rho(l(R)+1)}{\rho(l(R)+l_{\phi}(R))}\leq R^{-\phi(R)}\leq\frac{\rho(l(R))}{\rho(l(R)+l_{\phi}(R)+1)}.
\end{equation}

It is clear that both sides of (\ref{p1 eqn1}) are comparable to $2^{l_\phi(R)}$, and by $r_* > 0 $, both sides of (\ref{p1 eqn2}) are comparable to $\frac{\rho(l(R))}{\rho(l(R)+l_{\phi}(R))}$ respectively. This directly gives the desired argument.
\end{proof}

\subsection{General properties of generalised lower dimensions}

In this section, we focus on comparing generalized lower Assouad dimensions through dimension functions. For some n.b.d. spaces $F$, like self-similar sets or self-affine sets, we always have for any dimension functions $\phi$ and $\psi$ satisfying $\limsup_{R \to 0} \frac{\phi(R)}{\psi(R)} < 1$, 
$
\lad^\phi F \leq \lad^\psi F,
$
and for Ahlfors-regular sets $F$, we also have 
$
\lad^\phi F = \lad^\psi F.
$
However, by the dimension formulae of homogeneous Moran sets, the following examples show that the ordering of dimensions does not necessarily correspond to the ordering of their dimension functions.


\begin{example}\label{example 1}
Let $\phi,\psi$ be any two different dimension function satisfying $\liminf\limits_{R \to 0} \frac{\phi(R)}{\psi(R)} < 1 $, then there exist a compact set $M \subseteq \mathbb{R}$ such that 
\[
\lad^\phi M < \lad^\psi M.
\]
\end{example}

\begin{proof}
It follows from $\liminf\limits_{R \to 0} \frac{\phi(R)}{\psi(R)} < 1 $ that there exist some constant $\varepsilon > 0$ such that there exist a sequence $\{R_n \}_{n=1}^\infty$ satisfying that 
\[
\frac{\phi(R_n)}{\psi(R_n)} < 1-\varepsilon.
\]
By picking the elements of $\{R_n \}_{n=1}^\infty$, we assume that $\{R_n \}_{n=1}^\infty$ satisfying for all $n \geq 1$,
\[
R_{n+1} < \min\{R_n^{1+\psi(R_n)}/16, R_n^{2n} \}.
\]
We next apply the induction to construct the homogeneous Moran set $M(\mathbf{r}) \subseteq \mathbb{R}$ determined by $\mathbf{r}: = \{r_n \}_{n=1}^\infty$. Given $1  <  \alpha < \infty$, let $r_0 = 1, r_1 = R_1$, then by (\ref{l(R)}-\ref{l_phi(R)}), for any $n \geq 1$, we define
\[ r_k = 
\begin{cases}
2^{-\alpha} & l(R_n) + 1 \leq  k \leq l(R_n) + l_{\phi}(R_n); \\ 
2^{-1} & l(R_n) + l_{\phi}(R_n) + 1 \leq  k \leq l(R_{n + 1}).
\end{cases}
\]
Then it follows from Proposition \ref{homo_moran_formula} that
\[
\dim_{\rm L}^\phi M = \liminf_{n \to \infty}  \frac{l_\phi(R_n) \cdot \log 2}{\log \rho(l(R_n))/\rho(l(R_n)+ l_\phi(R_n))} = \frac{1}{\alpha}.
\]
Note that for all sufficiently large $n$, $\frac{\phi(R_n)}{\psi(R_n)} < 1-\varepsilon$, thus it follows from the definition of $l_\phi(R_n)$ and $l_\psi(R_n)$ that
\[
R_n^{-\phi(R_n)} \approx 2^{\alpha \cdot l_\phi(R_n)} \quad \text{ and } \quad  R_n^{-\psi(R_n)} \approx 2^{-l_{\psi}(R_n) - (\alpha-1) l_{\phi}(R_n)}.
\]
Hence there exist a constant $c_0 > 1$ such that for all sufficiently large $n$, $l_\psi(R_n) > c_0 \cdot l_\phi(R_n)$, which gives 
\begin{align*}
\lad^\psi F  & = \liminf_{n \to \infty} \dfrac{l_\psi(R_n) \log 2}{\log \rho(l(R_n))/\rho(l(R_n) + l_\psi(R_n))} \\
& = \liminf_{n \to \infty} \frac{l_\psi(R_n)}{l_\psi(R_n) + (\alpha - 1) \cdot l_\phi(R_n)} > \frac{1}{\alpha} = \lad^\phi F
\end{align*}
as desired.
\end{proof}

\begin{example}\label{example 2}
Let $\phi$ be a dimension function, let 
$\psi(R) = 3 \phi(R)/2$,
then there exist a compact set $M \subseteq \mathbb{R}$ such that 
\[
\lad^\phi M > \lad^\psi M.
\]
\end{example}
\begin{proof}
We now introduce the construction of the homogeneous Moran set $M \subseteq \mathbb{R} $ by induction. It follows from the choice of $\phi$ and $\psi$ that there exist a sequence $\{ R_n \}_{n=1}^\infty$ satisfying that for all $n\in \mathbb{N}$,
\[
R_{n+1} < \min\{R_n^{1+\psi(R_n)}/16, R_n^{2n} \}.
\]
Given $1 < \alpha < \infty$  and let $r_1 = R_1$, $l(R_1) = 1$. By (\ref{l(R)}-\ref{l_phi(R)}), for any $n \geq 1$, we define 
\[ r_k = 
\begin{cases}
2^{-\alpha} & l(R_n) + 1 \leq  k \leq l(R_n) + l_{\phi/2}(R_n); \\ 
2^{-1} & l(R_n) + l_{\phi/2}(R_n) + 1 \leq  k \leq l(R_n) + l_{\phi}(R_n); \\ 
2^{-\alpha} & l(R_n) + l_{\phi}(R_n) + 1 \leq  k \leq l(R_n) + l_{\psi}(R_n) ; \\
2^{-1} & l(R_n) + l_{\psi}(R_n) + 1 \leq  k \leq l(R_{n + 1}).
\end{cases}
\]
Similar to the proof of Example \ref{example 1}, by Proposition \ref{homo_moran_formula}, we obtain 
\[
\frac{\alpha + 2}{3\alpha} = \dim_{\rm L}^\psi M  < \dim_{\rm L}^\phi M = \frac{\alpha+1}{2\alpha}.
\]
The computation is similar to the proof of Example \ref{example 1}, thus we omit the details to readers.
\end{proof}

\section{Proof of main results.}

\subsection{Proof of Theorem 1: Equivalence}

The following result estimates the error between two different generalised lower dimension related to the quotient of the dimension functions.

\begin{proposition}\label{prop 2}
Suppose $\phi, \psi$ be two dimension functions, if there exist $\varepsilon > 0$ such that
\[
1 - \varepsilon \leq \liminf_{R \to 0} \frac{\phi(R)}{\psi(R)} \leq  \limsup_{R \to 0} \frac{\phi(R)}{\psi(R)} \leq 1+\varepsilon,
\]
then for any n.b.d. space $F$ with doubling constant $C$, it gives
\[
\left| \lad^\phi F - \lad^\psi F \right| \leq \varepsilon \cdot (1+ 2 \log_2 C + \varepsilon).
\]
\end{proposition}

\begin{proof}
It follows from the assumption that for all sufficiently small $0 < R < 1$, it gives
\[
1-\varepsilon \leq  \frac{\phi(R)}{\psi(R)} \leq 1+\varepsilon.
\]
Given sufficiently small $R$, if $\psi(R) > \phi(R)$, then by Lemma \ref{lower bound estimate}, it gives
\[
N_{R^{1+\psi(R)}}(B(x,R) \cap F) \leq  N_{R^{1+\phi(R)}}(B(x,R) \cap F) \cdot \sup_{y \in F} N_{R^{1+\psi(R)}}(B(y,R^{1+\phi(R)}) \cap F),
\]
thus, it follows from the doubling constant that 
\[
\sup_{y \in F} N_{R^{1+\psi(R)}}(B(y,R^{1+\phi(R)}) \cap F) \leq \left(R^{\phi(R) - \psi(R)}  \right)^{\log_2 C} = R^{(\phi(R) - \psi(R))\log_2 C},
\]
and by the definition of $\dim_{\rm L}^\phi F$ and $\dim_{\rm L}^\psi F$, it gives
\[
R^{-\psi(R) \cdot (\dim_{\rm L}^\psi F - \varepsilon) + (\psi(R) - \phi(R)) \cdot \log_2 C} \leq R^{-\phi(R) \cdot (\dim_{\rm L}^\phi F + \varepsilon)},
\]
as a consequence, we obtain that 
\[
-\psi(R) \cdot (\dim_{\rm L}^\psi F - \varepsilon) + (\psi(R) - \phi(R)) \cdot \log_2 C  \geq -\phi(R) \cdot (\dim_{\rm L}^\phi F +\varepsilon) 
\]
and then since $\dim_{\rm L}^\psi F \leq \log_2 C$ and $\psi(R) - \phi(R) \leq \varepsilon \cdot \psi(R)$ , we obtain
\[
\dim_{\rm L}^\psi F - \dim_{\rm L}^\phi F \leq \varepsilon \cdot (\varepsilon + 2\log_2 C + 1).
\]
 The result for the converse part holds by a similar proof.
\end{proof}
Now we go to the proof of Theorem 1.
\begin{proof}[Proof of Theorem 1]

(3) to (2): This directly holds. \\
(2) to (1): This follows from Example 1. \\
(1) to (3): This is a natural corollary of Proposition \ref{prop 2}.
\end{proof}

\subsection{Proof of Theorem 2: A tale of two generalised lower Assouad dimensions}

We first discuss the general property of rate windows. The following result shows the general bounds estimates of the rate windows for given dimension functions, and it is a generalisation of the continuity estimates of lower Assouad spectrum in \cite{CWC2020, HT2019}.

\begin{proposition}\label{prop 4}
Let $\phi$ be a dimension function. Let $F$ be a n.b.d. space and let $\varphi(\alpha) := \lad^{\phi_\alpha} F$. Then for any $0 < \alpha < \beta < \infty$,
\[
\left( \frac{1}{\alpha} - \frac{1}{\beta} \right) \varphi\left(\frac{\alpha \beta}{\beta - \alpha}\right)\leq \frac{1}{\alpha} \varphi(\alpha) - \frac{1}{\beta} \varphi(\beta) \leq c \left( \frac{1}{\alpha} - \frac{1}{\beta} \right).
\] 
\end{proposition}

\begin{proof}

We first discuss the lower bound estimates. For any $s_1 > \lad^{\phi_\alpha} F$, then there exist a monotonic decreasing sequence of $\{R_n\}_{n=1}^\infty$ satisfying $R_n < 1$ for all $n$, and tending to $0$ as $n$ tends to infinity, and a sequence $\{x_n \}_{n=1}^\infty \subseteq F$ such that for all $n$, let $(R_n)_{\alpha} = R_n^{1 + \phi(R_n)/\alpha}$ and $(R_n)_{\beta} = R_n^{1 + \phi(R_n)/\beta}$ respectively, then
\[
 M_{(R_n)_\alpha}(B(x_n,R_n) \cap F) \leq R_n^{-\phi_\alpha(R_n) \cdot s_1}.
\]

Let $s_2< \lad^{\phi_\beta} F$ and $s_3 < \lad^{\phi_{\alpha \beta/(\beta - \alpha)}} F$, then by the doubling property and $\phi$ monotonic, it gives
\begin{align*}
M_{16 (R_n/2)_\beta}(B(x_n,R_n/2) \cap F) & \approx M_{(R_n/2)_\beta}(B(x_n,R_n/2) \cap F)  \quad \text{(by doubling property)}  \\
&  \gtrsim (R_n/2)^{- \phi(R_n/2)s_2/\beta}\\
&  \gtrsim R_n^{- \phi(R_n/2)s_2/\beta} \\
&  \gtrsim R_n^{-  \phi(R_n) s_2/\beta} \quad  \text{(since  $\phi(R) > \phi(R/2) > 0$ )}
\end{align*}
Moreover, for all $y \in F$
\begin{align*}
M_{(R_n)_\alpha}(B(y,(R_n)_\beta) \cap F) & \gtrsim R^{-\phi(R)\cdot (1/\alpha - 1/\beta) s_3} \\ 
&  \gtrsim R^{-\phi(R)\cdot(1/(1/ (1/\alpha - 1/\beta))) s_3} \\
& \gtrsim R^{-\phi(R)\cdot(1/ (\alpha \beta/ (\beta - \alpha))) s_3}
\end{align*}
Note that for all sufficiently small $R > 0$ and any $0 < \alpha < \infty$,
\[
R - 2 R_\alpha > R/2.
\]
and for sufficiently small $R > 0$, $2^{\phi(R_n)/\alpha} < 8 $, and 
\[
0 \leq \dfrac{(R_n)_\alpha}{(R_n/2)_\alpha} \leq  2^{1+\phi(R_n/2)/\alpha} < 8.
\]
Hence by (\ref{doubling property}) and Lemma \ref{lower bound estimate}, we obtain
\begin{align*}
N_{(R_n)_\alpha}(B(x_n, R) \cap F) & \gtrsim N_{2(R_n)_\beta}(B(x_n, R_n - 2(R_n)_\beta) \cap F)  \\
& \qquad \cdot \inf_{y} M_{(R_n)_\alpha}(B(y,2(R_n)_\beta - (R_n)_{\alpha}) \cap F) \\
& \gtrsim N_{16(R_n/2)_\beta}(B(x_n, R_n/2) \cap F) \cdot \inf_{y} M_{(R_n)_\alpha}(B(y,(R_n)_\beta) \cap F),
\end{align*}
which implies that
\[
{- \frac{\phi(R)}{\beta} \cdot  s_2 - \frac{\phi(R)}{\alpha \beta/ (\beta - \alpha)} \cdot  s_3 \geq - \frac{\phi(R)}{\alpha} s_1 } 
\]
and
\[
\frac{1}{\beta} \cdot \varphi(\beta) + \frac{\beta - \alpha}{\alpha \beta} \cdot \varphi\left(\frac{\alpha \beta}{\beta - \alpha}\right) \leq 
\frac{1}{\alpha} \cdot \varphi(\alpha), 
\]
which is equivalent to
\[
\frac{1}{\alpha} \cdot \varphi(\alpha) -  \frac{1}{\beta} \cdot \varphi(\beta) \geq  \frac{\beta - \alpha}{\alpha \beta} \varphi\left(\frac{\alpha \beta}{\beta - \alpha}\right) .
\]

We next discuss the upper bound estimates. Let $s_1 < \lad^{\phi_\alpha} F $ and $s_2 > \lad^{\phi_\beta} F$, then there exist a monotonic decreasing sequence of $\{R_n\}_{n=1}^\infty$ satisfying $R_n < 1$ for all $n$, and tending to $0$ as $n$ tends to infinity, and a sequence $\{x_n \}_{n=1}^\infty \subseteq F$ such that for all $n$, 
\[
N_{(R_n)_\beta}(B(x_n,R_n) \cap F) \leq R_n^{-\phi_\beta(R_n) \cdot s_2}.
\]
and then, by Lemma \ref{upper bound estimate}, it gives
\[
R_n^{-\phi_\alpha(R_n) \cdot s_1} \leq R_n^{-\phi_\beta(R_n) \cdot s_2}  \cdot R_n^{-\phi(R_n) \cdot \left( \frac{1}{\alpha} - \frac{1}{\beta} \right) \cdot \log_2 M},
\]
thus there exists a constant $c$ related to $C$ such that
\[
\frac{1}{\alpha} \cdot \varphi(\alpha)  - \frac{1}{\beta} \cdot \varphi(\beta) \leq c \cdot \left( \frac{1}{\alpha} - \frac{1}{\beta} \right).
\]
\end{proof}


\begin{remark}
It follows from Proposition 3 that the generalised lower dimension of the rate windows is a continuous function.
\end{remark}

Now we gives the difference estimates between disjoint dimension function as follows.

\begin{proposition}\label{prop4}
Given a n.b.d. space $F$. Let $\psi$ be a dimension function. For any $\eps > 0$, there exist $\eta$ related to the doubling constant, $\eps, \psi$ such that for all other dimension functions $\phi$ satisfying $\limsup\limits_{R \to 0} \frac{\phi(R)}{\psi(R)} < \eta$, then we obtain
\[
\lqad^\psi F \geq \lad^\phi F - 2\eps.
\]
\end{proposition}

\begin{proof}
Let $s:=\lad^\phi F$. If $s = 0$, the discussion is trivial, thus we consider the case of $s > 0$, which suffices to prove for any sufficiently small $\varepsilon > 0$, then for any sufficiently small $R >0$ and $r < R^{1+\psi(R)}$, it gives 
\[
N_r(B(x,R) \cap F) \geq \left(\frac{R}{r}\right)^{s-2\varepsilon},
\]
and the lower bound estimates rely on the properties of $\lad^\phi F$ and the approximation of $r$ also relies on the construction of $\left\{ R_i^{1+\phi(R_i)} \right\}_{i=1}^\infty$.

Given dimension function $\psi$, it follows from the L'Hospital's rule that for any $\eps > 0$, there exist $\eta$ such that
\[
\eta \leq \frac{\left(1-\frac{1}{1+\psi(R)} \right) \varepsilon}{\psi(R) (s- \varepsilon)},
\]
which yields for any dimension function $\phi$ satisying $\limsup\limits_{R \to 0} \frac{\phi(R)}{\psi(R)} < \eta$, it gives
\begin{equation}\label{prop4.0}
\left(\frac{1}{1+\psi(R)} -1 \right) \varepsilon + \phi(R)(s- \varepsilon) \leq \left(\frac{1}{1+\psi(R)} -1 \right) \varepsilon + \eta\psi(R)(s- \varepsilon) <  0.
\end{equation}

Fix dimension function $\phi$ and $\psi$, for any sufficiently small $\eps > 0$, it follows from the dimension function $\phi$ that there exist $R' > 0$ such that for all sufficiently small $R  < R'$, $\phi(R) \leq 1$ and
\begin{equation}\label{prop4.1}
R^{1+\phi(R)} \leq \varepsilon \cdot R.
\end{equation}
Given $R > 0$, we define a sequence $\{R_n\}_{n=1}^\infty$ by induction. Let $R_1 = R$, then $R_2=R_1^{1+\phi(R_1)}$. If $R_k$ is defined, then $R_{k+1}=R_k^{1+\phi(R_k)}$. 

Hence, for any dimension function $\phi$ satisfying (\ref{prop4.0}) and any $r \leq R^{1+\psi(R)}$, we define $N$ be the maximal integer related to $r$ such that
\[
R_N > r \geq R_{N+1} = R_N^{1+\phi(R_N)}.
\]
Thus by (\ref{doubling property}), Lemma \ref{lower bound estimate} and Lemma \ref{lem3} such that for any $x \in F$, it gives 
\begin{align}\label{prop4.2}
\nonumber N_r(B(x,R) \cap F) & \approx M_r(B(x,R) \cap F) \geq M_{R_N}(B(x,R) \cap F) \\
&  \geq M_{R_1}(B(x,R/2) \cap F)  \cdots \inf_{x'}  M_{R_{N}}(B(x',R_{N}/2) \cap F).
\end{align}
Hence, by Lemma \ref{lem3}, Lemma \ref{lem4} and (\ref{prop4.1}), it gives
\begin{equation}\label{prop4.3}
M_{R_i}(B(x, R_{i-1}/2) \cap F ) \geq \left(\frac{R_{i-1}}{R_i}\right)^{s-\varepsilon}.
\end{equation}
This implies that for any $R_{N+1} \leq r \leq R_N$, by (\ref{prop4.2}) and (\ref{prop4.3}), we obtain
\begin{align}\label{prop4.4}
\nonumber N_r(B(x,R) \cap F) & \geq \left(\frac{R}{R_1}\right)^{(s-\varepsilon)}   \left(\frac{R_1}{R_2}\right)^{(s-\varepsilon)}   \cdots   \left(\frac{R_{N-1}}{R_N}\right)^{(s-\varepsilon)} \geq \left(\frac{R}{R_N}\right)^{(s-\varepsilon)}  \\
& \geq \left(\frac{R}{r}\right)^{(s-\varepsilon)}   \left(\frac{r}{R_N}\right)^{(s-\varepsilon)} 
\end{align}
Note that $\phi(R)$ is monotonic decreasing, it gives
\begin{equation}\label{prop4.5}
\left(\frac{r}{R_N}\right)^{s-\eps} \geq \left(\frac{R_{N+1}}{R_N}\right)^{s-\eps} \geq R_N^{\phi(R_N)\cdot (s-\eps)} \geq r^{\phi(R)\cdot (s-\eps)}.
\end{equation}
As a consequence, by 
$
R/r \geq  r^{\left(\frac{1}{1+\psi(R)} -1 \right)},
$
(\ref{prop4.0}), (\ref{prop4.4}) and (\ref{prop4.5}), we obtain
\begin{align*}
N_r(B(x,R) \cap F) \geq  \left(\frac{R}{r}\right)^{s-2\varepsilon} r^{\left(\frac{1}{1+\psi(R)} -1 \right) \varepsilon} r^{\phi(R) (s-\eps)}  \geq  \left(\frac{R}{r}\right)^{s-2\varepsilon},
\end{align*}
which gives the desired argument.
\end{proof}

Based on the proof of Proposition \ref{prop4}, we have the following corollary.
\begin{corollary}\label{coro4}
	 If there exist some $x\in F, R>0, \theta>0$, and $s>0$ such that $$N_{R^{1+\theta}}(B(x,R)\cap F)<R^{-\theta s},$$ then for every $\varepsilon<s$, there exists $\eta>0$ related to $\varepsilon$ such that for any $\theta^*<\eta\theta$, there exist $x^*\in F$ and $ R^*\in (R^{1+\theta},R) $ satisfying
	$$N_{{R^*}^{(1+\theta^*)}}\brackets{B(x^*,R^*)\cap F}<{R^*}^{-\theta^*(s+\varepsilon)}.$$
\end{corollary}

By the Corollary \ref{coro4}, we have the following technical lemmas for the variational principles and the recovery of lower Assouad dimension by dimension functions.
\begin{lemma}\label{lemma4-1}
	Let $F$ is a n.b.d. space. For any $\varepsilon\in(0,s)$, if there exists a sequence $\{(R_n, \theta_n, x_n)\}_{n=1}^{\infty}$ such that $x_n\in F$, $\lim\limits_{n\rightarrow\infty}R_n=0$, $\lim\limits_{n\rightarrow\infty}\theta_n\log \frac{1}{R_n}=\infty$ and
	\begin{equation}\label{eq4-2}
		N_{{R_n}^{1+\theta_{n}}}(B(x_n,R_n)\cap F)\leq({R_n})^{-\theta_n ({s+\frac{\varepsilon}{2^{n+1}}})},
	\end{equation}
then there exist a sequence $\{(R_n^*,\theta_n^*,x_n^*)\}_{n=1}^{\infty}$ satisfying $\{R_n^*\}_{n=1}^{\infty}$monotonically decreasing to $0$,    $\{\theta_n^*\}_{n=1}^{\infty}$ is non-increasing and $\lim\limits_{n\rightarrow\infty}\theta_n^*\log \frac{1}{R_n^*}=\infty$ , also
	\begin{equation}\label{eq4-3}
		N_{{R_n^*}^{1+\theta_n^*}}\brackets{B(x_n^*,R_n^*)\cap F}\leq \brackets{R_n^*}^{-\theta_n^* (s+\frac{\varepsilon}{2^n})}.
	\end{equation}
Moreover, suppose that for some dimension function $\phi(R)$, the inequality $\theta_n \geq \phi(R_n)$ holds for all sufficiently large $n$. Then (\ref{eq4-3}) holds for the related constructed sequence $\{(R_n^*, \theta_n^*)\}_{n=1}^{\infty}$ with $\theta_n^* \geq \phi(R_n^*)$. Furthermore, if the original sequence $\{\theta_n\}_{n=1}^\infty$ is unbounded, then for any $\theta > 0$, (\ref{eq4-3}) holds for some sequence $\{(R_n^*,\theta_n^*,x_n^*)\}_{n=1}^{\infty}$ where $\theta_n^* = \theta$ for all sufficiently large $n$.
\end{lemma}
\begin{proof}
We begin by considering the case when $\{\theta_n\}_{n=1}^\infty$ is unbounded. 
For any $\theta>0$ and $i \geq 1$, it follows from Corollary \ref{coro4} that for $\frac{\varepsilon}{2^{i+1}}$, we can find $\eta_i>0$, where $k>i$ and $\theta_k>\frac{\theta}{\eta_i}$, i.e. $\theta<\eta_i\theta_k$,  one can choose $(R_k^{(i)},\theta,x_k^{(i)}) $ satisfying $R_k^{(i)}\in (R_k^{1+\theta_k}, R_k)$ and
	$$N_{{R_k^{(i)}}^{1+\theta}}(B(x_k^{(i)},R_k^{(i)})\cap F)<{R_k^{(i)}}^{-\theta(s+\frac{\varepsilon}{2^{k+1}}+\frac{\varepsilon}{2^{i+1}})}
	<{R_k^{(i)}}^{-\theta(s+\frac{\varepsilon}{2^{i}})}.$$
Let 
$$
\mathcal{J}_i=\left\{k ~|~ \,k>i,\,\theta_k>\frac{\theta}{\eta_i}\right\},
$$
and
$$
	\mathcal{S}_i=\{(R_k^{(i)},\theta,x_k^{(i)}) ~|~\,k\in \mathcal{J}_i\}.
$$
Note that $\lim\limits_{k\rightarrow\infty}R_k^{(i)}=0$ and $\mathcal{J}_i$ are infinite since  $(\theta_n)_{n=1}^{\infty}$ is unbounded. For each $i$, we inductively pick an element $(R_k^{(i)},\theta,x_k^{(i)})$ from $\mathcal{S}_i$, denote by $(R_i^*,\theta,x_i^*)$, such that $R_{i+1}^*<\frac{R_{i}^*}{2}$. Thus, by letting $\theta_n^* = \theta$ for all sufficiently large $n$, we obtain the desired sequence.

Furthermore,  suppose that $\phi(R)$ is a dimension function such that $\theta_n \geq \phi(R_n)$ for all sufficiently large $n$. Notice that $\phi(R)$ is bounded, let $M$ be an upper bound, fix some $\theta^* > M$ and set $\theta_n^*=\theta^*$, thus, following the proof of the above case, the desired sequence also holds since $\theta_n^*=\theta^* > M \geq \phi(R_n^*)$ for all sufficiently large $n$.

In the case when $\{\theta_n\}_{n=1}^\infty$ is bounded, by passing to a subsequence, we may assume the sequence $\{\theta_n\}_{n=1}^\infty$ is monotone and converges to some $\theta^*$. If $\{\theta_n\}_{n=1}^\infty$ is non-increasing, it suffices to set $(R_n^*,\theta_n^*,x_n^*) = (R_n,\theta_n,x_n)$ for all $n \geq 1$. As $\{(R_n^*,\theta_n^*,x_n^*)\}_{n=1}^{\infty}$ is a subsquence of $\{(R_n,\theta_n,x_n)\}_{n=1}^{\infty}$, it automatically satisfies (\ref{eq4-3}), 
also $\theta_n^*\geq\phi(R_n^*)$, thus, the proof for this case is complete.  

Otherwise, if $\{\theta_n\}_{n=1}^\infty$ is increasing and converges to some $\theta^*>0$, we can set $(R_n^*,\theta_n^*,x_n^*)=(R_n,\theta^*,x_n)$ for all $n$. Hence,
\begin{equation}\label{eq4-2-1}
	\begin{aligned}
			N_{{R_n^*}^{1+\theta_n^*}}\brackets{B(x_n^*,R_n^*)\cap F}
			&=	N_{{R_n}^{1+\theta_{n}}}(B(x_n,R_n)\cap F)\\
			&\leq({R_n})^{-\theta_n ({s+\frac{\varepsilon}{2^{n+1}}})}\\
			&=({R_n}^*)^{-\theta^*[\frac{\theta_n}{\theta^*}(s+\frac{\varepsilon}{2^{n+1}})]}
	\end{aligned}
\end{equation}
Observe that for every $n$, we can find some $n^{\prime}$ such that for any $k>n^{\prime}$, we have 
$$
\frac{\theta_k}{\theta^*}\cdot \left( s+\frac{\varepsilon}{2^{k+1}} \right) < s+\frac{\varepsilon}{2^n}.
$$
Therefore, by choosing a suitable subsequence, we obtain a sequence, still denoted by $\{(R_n,\theta_n^*,x_n)\}_{n=1}^{\infty}$ with $\theta_n^*=\theta^*$ for all sufficiently large $n$, that fulfills (\ref{eq4-3}) and $\theta_n^* \geq \theta_n \geq \phi(R_n)=\phi(R_n^*) $. This completes the proof.
\end{proof}

\begin{lemma}\label{lemma4-2}
	Suppose $F$ is a n.b.d. space, and $\{(R_n, \theta_n, x_n)\}_{n=1}^{\infty}$ is a sequence where $\{R_n\}_{n=1}^{\infty}$ is monotonic decreasing to $0$, $\{\theta_n\}_{n=1}^\infty$ is non-increasing,	$\{\theta_n\log\frac{1}{R_n}\}_{n=1}^{\infty}$ monotinically increases to $+\infty$. For any $\varepsilon\in(0,s)$ and for all sufficiently large $n \geq 1$, it gives
\begin{equation}\label{eq-lem4-2-1}
	N_{R_n^{1+\theta_n}}(B(x_n,R_n)\cap F)\leq R_n^{-(s+\frac{\varepsilon}{2^n})\theta_n},
\end{equation}
then there exists a dimension function $\phi(R)$ satifsying $\phi(R_n)=\theta_n$ for all $n \geq 1$, and then $\dim_{\rm L}^{\phi}F\leq s$. Moreover, for any dimension function $\phi^*(R)$ satisfying $\phi^*(R_n)=\theta_n$, we always have $\phi(R)\geq \phi^*(R)$,  which yields $\phi(R)$ is a maximal element with respect to the given conditions.
\end{lemma}

\begin{proof}
If for some integer $n$, $\theta_{n+1}=\theta_n$, it suffices to define $\phi(R)=\theta_n$ on $[R_{n+1},R_n]$. This gives $\phi(R)\log \frac{1}{R}$ is strictly decreasing on $[R_{n+1},R_n]$. 
Thus, without loss of generality, we assume that $\theta_{n+1}<\theta_n$. By assumptions, for any $n \geq 1$, $\theta_{n+1} \log \frac{1}{R_{n+1}} > \theta_{n} \log \frac{1}{R_{n}}$, i.e., $R_{n+1}^{\theta_{n+1}}<R_n^{\theta_n}$, and we can find some $R_n'\in(R_{n+1},R_n)$, such that $\theta_{n} \log \frac{1}{R_n^\prime} = \theta_{n+1} \log \frac{1}{R_{n+1}}$, i.e., $R_n'^{\theta_{n}}=R_{n+1}^{\theta_{n+1}}$.

We first introduce the construction of the dimension function. For any $R \in [R_n',R_n]$, we define $\phi(R)=\theta_{n}$, then  $\phi(R) \log \frac{1}{R}$ is decreasing on $[R_n',R_n]$. For $R \in [R_{n+1},R_n']$, let $\phi(R)=\theta_{n+1}\cdot \frac{\log R_{n+1}}{\log R}$, then $\phi(R)\log \frac{1}{R}$ is a constant ${\theta_{n+1}} \log \frac{1}{R_{n+1}}$, this gives $\phi(R)$ is increasing. By the definition of dimension, we see that $\phi(R)$ is a dimension function satisfing $\phi(R_n)=\theta_n$ for all $n$, i.e.,
\[
\phi(R) =
\begin{cases}
\theta_n, & R \in [R_n',R_n]; \\
\theta_{n+1}\cdot \frac{\log R_{n+1}}{\log R}, & R \in [R_{n+1},R_n'].
\end{cases}
\]
 
 We now prove that for any dimension function $\phi^*$, $\phi(R)\geq \phi^*(R)$ where $\phi^*(R_n)=\theta_n$ for all $n \geq 1$. Given any $n \geq 1$, for any $R\in[R_n',R_n]$, according to the definition of dimension function, it gives
 \begin{equation}\label{lemma6-1}
 	\phi^*(R)\leq \phi^*(R_n)=\phi(R_n)=\phi(R).
 \end{equation}
While $R\in[R_{n+1},R_n']$, we have
\begin{equation}\label{lemma6-2}
	\phi^*(R)\log R\geq \phi^*(R_{n+1})\log {R_{n+1}}=\theta_{n+1}\log R_{n+1}= \phi(R)\log R.
\end{equation}
This gives $\phi^*(R)\leq \phi(R) $ for any $R \in [R_{n+1},R_n]$.

It remains to prove that $\dim_{\rm L}^{\phi} F \leq s$. It follows from the definition of generalised lower dimension that
\begin{align*}
	\dim_{\rm L}^{\phi}F & \leq \liminf\limits_{R \to 0}\frac{\inf\limits_{x\in F}\log N_{R^{1+\phi(R)}}\brackets{B(x,R)\cap F}}{-\phi(R)\log R} \\
	& \leq \liminf\limits_{R \to 0}\frac{\inf\limits_{x\in F}\log N_{R_n^{1+\phi(R_n)}}\brackets{B(x,R_n)\cap F}}{-\phi(R_n)\log R_n}\leq s.
\end{align*}
\end{proof}

\begin{remark}
Similar to the proof of Lemma \ref{lemma4-2}, for the sequence $\{(R_n, \theta_n, x_n) \}_{n=1}^\infty$ in Lemma \ref{lemma4-2}, we can also define the dimension function $\widetilde{\phi}(R)$ such that $\widetilde{\phi}(R)$ is the mininal element satisfying $\widetilde{\phi}(R_n) = \theta_n$ for all $n \geq 1$. In other words, for any dimension function $\phi(R)$ satisfying $\phi(R_n) = \theta_n$ for all $n \geq 1$, we always have $\widetilde{\phi}(R) \leq \phi(R)$. Hence we just need to consider 
\[ \widetilde{\phi}(R) =  
\begin{cases}
\frac{\theta_n \log R_n}{\log R} & R \in [\widetilde{R_n}, R_n] \\
\theta_{n+1} & R \in [R_{n+1}, \widetilde{R_n}] 
\end{cases}
\]
where $\widetilde{R_n}$ satisfies $\theta_{n+1} \log \widetilde{R_n} = \theta_n \log R_n$.
\end{remark}

We now turn to prove the Theorem \ref{THM_variational_principle}, which gives a variational principle between two different generalised lower dimensions.

\begin{proof}[Proof of Theorem \ref{THM_variational_principle}]
	If $\lim\limits_{R \to 0} \phi(R) > 0$, then we choose $\phi(R) = -1 + 1/\theta$ for some $\theta \in (0,1)$, and it follows from Theorem \ref{THM_equivalence} and \cite{CWC2020} and \cite{HT2019} that 
	\[
	\underline{\dim}_{\rm L}^\phi F = \underline{\dim}_{\rm L}^\theta F =  \inf_{0 < \theta' \leq \theta} \dim_{\rm L}^{\theta'} F,
	\]
	where $\underline{\dim}_{\rm L}^\theta F $ and $\dim_{\rm L}^{\theta} F$ are the quasi-lower Assouad spectrum and lower Assouad spectrum of $\theta \in (0,1)$. 
	
	Thus, we always assume that  $\lim\limits_{R \to 0} \phi(R) = 0$. For any $\alpha \in (0,1)$, and any $0 < R < 1$, $\phi(R) \leq \phi(R)/\alpha$, thus it follows from the definition of both $\underline{\dim}_{\rm L}^\phi F $ and $\dim_{\rm L}^{\phi_\alpha} F $ that
	\[
	\underline{\dim}_{\rm L}^\phi F  \leq  \inf_{0 < \alpha < 1} \dim_{\rm L}^{\phi_\alpha} F 
	\]
	For the converse inequality, let $s := \underline{\dim}_{\rm L}^\phi F$, if $s = 0$, then the discussion is trivial. Thus for the rest of this proof, we consider the case of $s > 0$, and it suffices to prove 
	\[
	\inf_{0 < \alpha < 1} \dim_{\rm L}^{\phi_\alpha} F  \leq \underline{\dim}_{\rm L}^\phi F. 
	\]
	
	Given sufficiently small $\varepsilon > 0$, it follows from the definition of $\underline{\dim}_{\rm L}^\phi F$ that there exist a sequence $\{ x_n \}_{n=1}^\infty \subseteq F$, and $\{(r_n, R_n) \}_{n=1}^\infty $ satisfying that $r_n \leq R_n^{1+\phi(R_n)}$ and
	\[
	N_{r_n}(B(x_n,R_n) \cap F) \leq \left(\frac{R_n}{r_n} \right)^{s+\frac{\varepsilon}{2^{n+1}}}.
	\]
	Fix $n$, we define $\theta_n$ be the solution of 
	\[
	r_n = R_n^{1+\theta_n},
	\]
	thus $\phi(R_n)\leq \theta_n$.
	According to Lemma \ref{lemma4-1}, we can find a sequence $\{(R_n^{*},\theta_n^{*},x_n^{*})\}_{n=1}^{\infty}$ such that $\{\theta_n^{*}\}_{n=1}^{\infty}$ is non-increasing and $\phi(R_n^{*})\leq \theta_n^{*}$,
	also,
	$$N_{{R_n^{*}}^{1+\theta_{n}^*}} (B(x_n, R_n^*) \cap F)\leq {R_n^{*}}^{-\theta_n^{*}(s+\frac{\varepsilon}{2^{n}})}$$
	 
	Combined with Lemma \ref{lemma4-2}, there exists a dimension function $\widetilde{\phi}(R)$ such that 
	 $\widetilde{\phi}({R_n^{*}})=\theta_{n}^{*}$, therefore $\dim_{\mathrm L}^{\widetilde{\phi}}F\leq s$ and $\phi(R_n^*)\leq \theta_{n}^*=\widetilde{\phi}({R_n^{*}})$. 
	 
	We denote $$b_n=\frac{\phi({R_n^*})}{\widetilde{\phi}({R_n^*})},$$ then $0<b_n\leq 1$. By passing to a subsequence, 
	we assume that $\{b_n\}_{n=1}^{\infty}$ is monotomic and donote $b_0$ by the limit of  $\{b_n\}_{n=1}^{\infty}$.
	\begin{enumerate}
			\item If $b_0=0$, it follows from Proposition \ref{prop4} that
			$$\dim_{\mathrm{L}}^{\phi}F\leq\dim_{\mathrm{L}}^{\widetilde{\phi}}F\leq s=\underline{\dim}_{\mathrm{L}}^{\phi}F,$$ 
			therefore
			 \[
			\inf_{0 < \alpha < 1} \dim_{\rm L}^{\phi_\alpha} F  \leq \underline{\dim}_{\rm L}^\phi F. 
			\]
	       \item If $0<b_0<1$, in this case,
	        $\lim\limits_{n\rightarrow \infty}\frac{\phi({R_n^*})/b_0}{\widetilde{\phi}({R_n^*})}=1$, 
	        thus $\dim_{\mathrm{L}}^{\phi_{b_0}}F\leq s$, in this case, we also have
	        \[
	        \inf_{0 < \alpha < 1} \dim_{\mathrm L}^{\phi_\alpha} F  \leq \underline{\dim}_{\mathrm L}^\phi F. 
	        \]
	\end{enumerate}
\end{proof}

Directly by Theorem \ref{THM_variational_principle}, we obtain the natural corollary as follows.

\begin{corollary}
Given n.b.d. space $F$ and dimension function $\phi(R)$, there exist a dimension function $\psi(R)$ such that
\[
 \underline{\dim}_{\rm L}^\phi F = \underline{\dim}_{\rm L}^\psi F = \dim_{\rm L}^\psi F.
\]
\end{corollary}

\subsection{Proof of Theorem \ref{THM_interpolation_positive}: Interpolation of generalised lower dimensions}

In this part, we study the interpolation of generalised lower Assouad type dimensions, and give the proof of Theorem \ref{THM_interpolation_positive}. We first discuss the recovery of lower Assouad dimension by dimension functions. Note that using the similar analogous of Theorem  \ref{THM_variational_principle}, we obtain the  different proof on the general recovery of lower Assouad dimension as follows.

\begin{proposition}\label{low1}
Given $F \subseteq \mathbb{R}^d$. Let $s := \dim_{\mathrm{L}}F$, then there exist a dimension function $\phi(R)$ such that $\dim_{\mathrm{L}}^{\phi}F=s$.
\end{proposition}

\begin{proof}
Given sufficiently small $\varepsilon > 0$, it follows from the definition of $\dim_{\mathrm L} F$ that there exist a sequence $\{ x_n \}_{n=1}^\infty \subseteq F$, and $\{(r_n, R_n)\}_{n=1}^\infty $ satisfying that $\frac{r_n}{R_n}$ tends to $0$ and
\[
N_{r_n}(B(x_n,R_n) \cap F) \leq \left(\frac{R_n}{r_n} \right)^{s+\frac{\varepsilon}{2^{n+1}}}.
\]
Fix $n$, we define $\theta_n$ be the solution of 
\[
r_n = R_n^{1+\theta_n}.
\]
According to Lemma \ref{lemma4-1}, we can find a sequence $\{(R_n^{*},\theta_n^{*},x_n^{*})\}_{n=1}^{\infty}$ such that $\{\theta_n^{*}\}_{n=1}^{\infty}$ is non-increasing and $\phi(R_n^{*})\leq \theta_n^{*}$. Also,
$$N_{{R_n^{*}}^{1+\theta_{n}^*}}(B(x_n^*, R_n^*) \cap F) \leq {R_n^{*}}^{-\theta_n^{*}(s+\frac{\varepsilon}{2^{n}})}.$$

Combined with Lemma \ref{lemma4-2}, there exists a dimension function $\widetilde{\phi}(R)$ such that
$\widetilde{\phi}({R_n^{*}})=\theta_{n}^{*}$, thus we have $\dim_{\mathrm L}^{\widetilde{\phi}}F\leq s$. The converse part is trivial.
\end{proof}

According to Lemma \ref{lemma4-1}, Lemma \ref{lemma4-2} and Proposition \ref{low1}, we have a proof of Theorem \ref{THM_lower_dim} as follows.

\begin{proof}[Proof of Theorem \ref{THM_lower_dim}]
Let $s := \dim_{\mathrm{L}}F$. By the assumption of $s$ and combined with Lemma \ref{lemma4-1}, for any $\varepsilon > 0$, we have a sequence $\{(R_n, \theta_n, x_n)\}_{n=1}^{\infty}$ such that $\{R_n \}_{n=1}^\infty$ is monotonic decreasing to $0$, $\{\theta_n \}_{n=1}^\infty$ is non-increasing and $\lim_{n \to \infty} \theta_n \log \frac{1}{R_n} = \infty$. Besides, for any $n \geq 1$, it gives 
\[
N_{R_n^{1+\theta_n}}(B(x_n, R_n ) \cap F) \leq R_n^{-\theta_n(s+\frac{\varepsilon}{2^{n+1}})}.
\]

If there exist some constant $\theta > 0$ such that $\lim_{n \to \infty} \theta_n = \theta$, in this case, $\dim_{\rm L}^\theta F  = \dim_{\rm L} F$, then the result is trivial. Without loss of generality, we always assume that $\lim_{n \to \infty} \theta_n = 0$. If for a subsequence $\{n_k \}_{k = 1}^\infty \subseteq \mathbb{N}$ such that $\theta_{n_k} \leq \Phi(R_{n_k})$, let $\Psi(R)$ be the minimal dimension function such that $\Psi(R_{n_k}) = \theta_{n_k}$, then $\Psi(R_{n_k}) = \theta_{n_k} \leq \Phi(R_{n_k})$. Since $\Psi(R)$ is the minimal one, we obtain that $\Psi(R) \leq \Phi(R)$ for all $R  > 0$. This also gives $\dim_{\mathrm L}^\Psi F = s$.

If $\{\frac{\theta_n}{\Phi(R_n)} \}_{n=1}^\infty$ is unbounded, then by passing to a subsequence, we assume that $\{\frac{\theta_n}{\Phi(R_n)} \}_{n=1}^\infty$ is monotonic increasing to infinity. This gives $\lim_{n \to \infty} \frac{\Phi(R_n)}{\theta_n} = 0$. Thus, combining Lemma \ref{lemma4-2} and Proposition \ref{low1}, the result naturally holds.

Otherwise, if there exist a constant $C > 1$ such that for all $n \geq 1$, $1 \leq \frac{\theta_n}{\Phi(R_n)} \leq C$, since $\lim_{n \to \infty} \theta_n \log \frac{1}{R_n} = \infty$, by passing to a subsequence again, we assume that $\theta_n \log \frac{1}{R_n} >  n \theta_{n-1} \log \frac{1}{R_{n-1}}$, i.e., 
\[
\frac{\theta_n}{n} \log \frac{1}{R_n} >  \theta_{n-1} \log \frac{1}{R_{n-1}}.
\]
By the techniques of Lemma \ref{lemma4-1}, there exist a sequence $\{(R_n^{*}, \theta_n^*, x_n^{*})\}_{n=1}^{\infty}$ satisfying $\{R_n^* \}_{n=1}^\infty$ is monotonic decreasing to $0$, $x_n^* \in F$ for all $n \geq 1$, $\lim_{n \to \infty} \frac{\theta_n^*}{\theta_n} = 0$, $\lim_{n \to \infty} \theta_n^* \log \frac{1}{R_n^*} = \infty$ and 
\[
N_{{R_n^{*}}^{1+\theta_n^*}}(B(x_n^*, R_n^*) \cap F) \leq {R_n^{*}}^{-\theta_n^*(s+\frac{\varepsilon}{2^{n}})}.
\]
Then let $\Psi(R)$ be the mininal dimension function satisfying $\Psi(R_n^*) = \theta_n^*$, then the desired argument holds. This gives the complete proof. 
\end{proof}

We now turn to the interpolation on the region $(\dim_{\rm L }F, \dim_{\rm qL }F]$. We first give some notations. For any $x > 0$ and $0 < y < x$, we first define the local complexity function $\omega(x,y)$ by 
\begin{equation}\label{omega x y}
\omega(x, y) = \inf_{\substack{\xi \in F}} \frac{\log N_{y}(B(\xi,x) \cap F)}{\log (x/y)}.
\end{equation}
Given any positive continuous functions $f(t), f_1(t), f_2(t): \mathbb{R}^+ \to \mathbb{R}^+ $, we also define the regions $\mathcal{D}(R,f(t))$ and $\mathcal{D}(R,f_1(t),f_2(t))$ determined by 
\begin{align}\label{regions}
	\mathcal{D}(R,f(t)) & = \bigbrackets{(t,y) \mathdot 0 < t \leq R, \, 0 < y \leq f(t) }, \\
	\mathcal{D}(R,f_1(t),f_2(t)) & =	\mathcal{D}(R,f_1(t))\setminus 	\mathcal{D}(R,f_2(t)).
\end{align}

 The following technical claim is applicable to discuss the interpolation of lower Assouad-type dimensions under generalised lower Assouad dimensions.

\begin{claim}\label{cm3}
	For any $s \in (\dim_{\rm L }F, \dim_{\rm qL }F]$ and sufficiently small $\varepsilon > 0$, if there exists
a dimension function $\phi(x)$, and sequence $\{R_i\}_{i=1}^\infty$ decreasing to $0$, such that for any $i \geq 1$, it gives 
	\begin{enumerate}
		\item  For any $(x,y)\in\mathcal{D}(R_i, t^{1+\phi(t)})$, $\omega(x,y)\geq s$,
		\item 	$\omega(R_i, R_i^{1+\phi(R_i)})\leq s+ \frac{\varepsilon}{2^i}$,
	\end{enumerate}
	then we have $\dim_{\mathrm L}^{\phi} F =\underline{\dim}_{\mathrm L}^{\phi} F =s$.
\end{claim}

\begin{proof}
This result directly follows from the definition of generalised lower Assouad dimensions.
\end{proof}

The following proposition provides an interpolation between the lower and quasi lower dimensions.
\begin{proposition}\label{THM_interpolation_2}
	Let $s \in (\dim_{\rm L} F, \dim_{\rm qL} F]$, then there exist a dimension function $\phi(R)$ such that $s = \dim_{\rm L}^\phi F = \underline{\dim}_{\rm L}^\phi F$.
\end{proposition}
\begin{proof}
For any $s\leq\dim_{\rm qL} F$, if there exists $\theta>0$ such that $\dim_{\rm L}^{\theta} F =s$, then let $\phi(R)=\theta$ for sufficiently small $R > 0$, it yields the desired result. 
Otherwise, given $\theta_0=1$, we can find sufficiently small $R_1<1$ such that for any $(x,y)\in\mathcal{D}(R_1, t^{1+\theta_0})$, we have $\omega(x,y)\geq s$.
Let $k_1=R_1^{\theta_0}$ and define 
\begin{align}\label{prime1}
\notag {R^{\prime}_1}=\sup\{x & \mathdot \text{ for any } 0< t <R_1 \text{ and } \\  
& (x,y)\in\mathcal{D}(R_1,k_1t,t^{1+\theta_0})\setminus\mathcal{D}(x,k_1t,t^{1+\frac{\log k_1}{\log x}}),  \omega(x,y)\geq s,\}
\end{align}
let $\theta_1=\frac{\log k_1}{\log {R^{\prime}_1}}$, (\ref{prime1}) implies that, 
\begin{equation}\label{geq1}
	\omega(x,y)\geq s,\quad \text{ for any } (x,y)\in\mathcal{D}(R_1,k_1t,t^{1+\theta_0})\setminus\mathcal{D}(x,k_1t,t^{1+\theta_1}).
\end{equation}
Also we can find $R_1^*\leq {R^{\prime}_1}$ such that
\begin{equation}\label{leq1}
	\omega(R_1^*,{(R_1^*)}^{1+\theta_1})\leq s + \frac{\varepsilon}{2}.
\end{equation}
Denote $R_2=\frac{1}{2} R_1^*$, $k_2=R_2^{\theta_1}$, and define
\[
\phi(R) = 
\left\{
\begin{array}{ll}
	\frac{\log k_1}{\log R}& {R^{\prime}_1}<R\leq R_1; \\
	\theta_1 & R_2<R\leq {R^{\prime}_1}.
\end{array}
\right.
\]
(\ref{geq1}) implies 
\begin{equation}
\omega(x,y)\geq s,\qquad \text{when } R_2<x\leq R_1	\text{ and }   0<y<x^{1+\phi (x)}.
\end{equation}
Recursively, if ${R^{\prime}_n}$, $R_n^*$ and $\theta_n$ have been defined, Let $R_{n+1}=\frac{1}{2}R_n^*$, $k_{n+1}=R_{n+1}^{\theta_n}$, we define
\begin{align}\label{prime}
\notag {R^{\prime}_{n+1}}=\sup\{ x \mathdot &  \text{ for any } 0<t<R_{n+1} \text{ and }\\
& (x,y)\in\mathcal{D}(R_{n+1},k_{n+1}t,t^{1+\theta_n})\setminus\mathcal{D}(x,k_{n+1}t,t^{1+\frac{\log k_n}{\log x}}), w(x,y)\geq s\}.
\end{align}
Let $\theta_{n+1}=\frac{\log k_{n+1}}{\log {R^{\prime}_{n+1}}}$, then there exist  $R_{n+1}^* < R_{n+1}$ satisfying
\[(R_{n+1}^*, {(R_{n+1}^*)}^{1+\theta_{n+1}})\in \mathcal{L}({R^{\prime}_{n+1}} ,t^{1+\theta_{n+1}})\subset\mathcal{D}({R_n,k_n t,t^{1+\theta_{n+1}}}),\]
and 
\[\omega(R_{n+1}^*, {(R_{n+1}^*)}^{1+\theta_{n+1}})\leq s+\frac{\varepsilon}{2^{n+1}}.\]
Thus, for any $0 < R \leq R_1$, we define
\[
\phi(R) = 
\left\{
\begin{array}{ll}
	\frac{\log k_n}{\log R}& {R^{\prime}_{n+1}}<R\leq R_n,  \\
	\theta_{n+1} & R_{n+1}<R\leq {R^{\prime}_{n+1}}.
\end{array}
\right.
\]
Details see Figure \ref{low dim Rphi 2}.

\begin{figure}[H]
\centering
\begin{tikzpicture}[font=\tiny, >=Stealth, scale=0.45]

  \pgfmathsetmacro{\Rn}{18}        
  \pgfmathsetmacro{\m}{0.82}       
  \pgfmathsetmacro{\pOut}{2}       
  \pgfmathsetmacro{\pDas}{1.8}       
  \pgfmathsetmacro{\pMid}{1.5}     

  \pgfmathsetmacro{\rA}{0.40}  
  \pgfmathsetmacro{\rB}{0.80}  
  \pgfmathsetmacro{\rC}{0.90}  
  \pgfmathsetmacro{\rM}{0.7}  

  \pgfmathsetmacro{\xA}{\rA*\Rn}   
  \pgfmathsetmacro{\xB}{\rB*\Rn}   
  \pgfmathsetmacro{\xC}{\rC*\Rn}   
  \pgfmathsetmacro{\Rm}{\rM*\Rn}   
  \pgfmathsetmacro{\yRn}{\m*\Rn}

  \pgfmathsetmacro{\yAonInner}{ \m*\xA*pow(\xA/\xC, \pDas-1) }
  \coordinate (RnplusOneOnInner) at (\xA,\yAonInner); 

  \draw[-Stealth] (0,0) -- (20,0) node[right] {$x$};
  \draw[-Stealth] (0,0) -- (0,14) node[above] {$y$};

  \draw[very thick, loosely dashed, gray] (0,0) -- (\xC,{\m * \xC});

  \draw[thick, ] (\xC,{\m * \xC}) -- (\Rn,\yRn);

  \fill[gray!20] 
    plot[domain=\xA:\Rn, samples=260]
      (\x, { \m*\x*pow(\x/\Rn, \pOut-1) })
    -- 
    plot[domain=\xC:\xA, samples=260] 
      (\x, { \m*\x*pow(\x/\xC, \pDas-1) })
    -- cycle;

  \draw[very thick, loosely dashed, line cap=round, blue!50]
    plot[domain=0:\Rn, samples=260]
      (\x, { \m*\x*pow(\x/\Rn, \pOut-1) });

  \draw[thick, line cap=round]
    plot[domain=\xA:\xC, samples=260]
      (\x, { \m*\x*pow(\x/\xC, \pDas-1) });

  \pgfmathsetmacro{\xBB}{\m*\xB*pow(\xB/\xC, \pDas-1)}   

  \draw[dashed] (\xA,0) -- (\xA,\yAonInner) node[below=4pt] at (\xA+0.6,0) {$R_{n+1}$};
  \draw[dashed] (\xB,0) -- (\xB,\xBB)   node[below=4pt] at (\xB,0) {$R_n^{*}$};
  \draw[dashed] (\xC,0) -- (\xC,{\m*\xC+0.15}) node[below=4pt] at (\xC,0) {$R_n'$};
  \draw[dashed] (\Rn,0) -- (\Rn,{\yRn+0.2}) node[below=4pt] at (\Rn,0) {$R_n$};

  \pgfmathsetmacro{\yDot}{ \m*\xB*pow(\xB/\xC, \pDas-1) }
  \fill  [red] (\xB,\yDot) circle (3pt);

  \pgfmathsetmacro{\Rk}{\xA}             
  \pgfmathsetmacro{\xAone}{\rA*\Rk}      
  \pgfmathsetmacro{\xBone}{\rB*\Rk}      
  \pgfmathsetmacro{\xCone}{\rC*\Rk}      
  \pgfmathsetmacro{\Rmon}{\rM*\Rk}       

  \pgfmathsetmacro{\yEnd}{\yAonInner}
  \pgfmathsetmacro{\sRB}{ \yEnd / (\m*\Rk) }  

  \pgfmathsetmacro{\tMid}{0.5} 
  \pgfmathsetmacro{\xMeet}{ (1-\tMid)*\xAone + \tMid*\xCone }

  \pgfmathsetmacro{\yRedAtXmeet}{ \sRB * \m * \xMeet }

  \fill[gray!20] 
    plot[domain=\xAone:\xA, samples=260]
        (\x, { \m*\x*pow(\x/\Rn, \pOut-1) })  
    -- 
    plot[domain=\xA:\xAone, samples=260] 
  (\x,{ \sRB * \m * \x * pow(\x/\Rk, \pOut-1) })
    -- cycle;

  \draw[thick, black, line cap=round]
    plot[domain=\xCone:\Rk, samples=200]
      (\x, { \sRB * \m * \x });

  \draw[thick, dashed, line cap=round, blue!50]
    plot[domain=0:\Rk, samples=200]
      (\x, { \sRB * \m * \x * pow(\x/\Rk, \pOut-1) });

  \draw[thick, line cap=round]
    plot[domain=\xAone:\xCone, samples=220]
      (\x, { \sRB * \m*\x*pow(\x/\xCone, \pDas-1) });

  \pgfmathsetmacro{\yDDot}{{ \sRB * \m*\xBone*pow(\xBone/\xCone, \pDas-1) }}
  \fill [red] (\xBone,\yDDot) circle (3pt);

  \pgfmathsetmacro{\yAoneInner}{ \m*\xAone*pow(\xAone/\xCone, \pDas-1) }
  \pgfmathsetmacro{\yBone}{\m*\xBone}
  \pgfmathsetmacro{\yCone}{\m*\xCone}
  \pgfmathsetmacro{\yCOone}{\sRB * \m*\xAone*pow(\xAone/\xCone, \pDas-1)}
  \pgfmathsetmacro{\yRedAtXCome}{ \sRB * \m*\xCone*pow(\xCone/\xCone, \pDas-1) }

  \draw[dashed] (\xAone,0) -- (\xAone,\yCOone)    node[below=4pt] at (\xAone,0) {$R_{n+2}$};
  \draw[dashed] (\xBone,0) -- (\xBone,\yDDot)         node[below=4pt] at (\xBone-0.6,0) {$R_{n+1}^{*}$};
  \draw[dashed] (\xCone,0) -- (\xCone,\yRedAtXCome)  node[below=4pt] at (\xCone,0) {$R_{n+1}'$};

  \coordinate (P1) at (\xB,\yDot) ;
  \node (txt) at (\xB -0.1 ,\yDot + 5) {$\omega(R_{n}^*, {(R_{n}^*)}^{1+\theta_{n}})\leq s+\varepsilon_{n}$};
  \draw[->, shorten <=2pt] (P1) -- (txt);

  \coordinate (P2) at (\xBone,\yDDot);
  \node (txt) at (\xBone -0.1,\yDDot + 7) {$\omega(R_{n+1}^*, {(R_{n+1}^*)}^{1+\theta_{n+1}})\leq s+\varepsilon_{n+1}$};
  \draw[->, shorten <=2pt] (P2) -- (txt);

  \fill (RnplusOneOnInner) circle (1.3pt);

  \node[
    draw, rounded corners, fill=white, inner sep=3pt, fill opacity=0.85,
    row sep=1.5pt, column sep=2mm, nodes={anchor=west}, align = left,
  ] at (24,14) {
    \tikz{\draw[very thick, loosely dashed, gray] (0,0)--(8mm,0);} $y=k_n R$ \\
    \tikz{\draw[thick, dashed, gray] (0,0)--(8mm,0);} $y = k_{n+1} R$ \\
    \tikz{\draw[very thick, loosely dashed, blue!50] (0,0)--(8mm,0);} $y = R^{1+\theta_{n-1}}$ \\
    \tikz{\draw[thick, dashed, blue!50] (0,0)--(8mm,0);} $y = R^{1+\theta_{n}}$ \\
    \tikz{\draw[thick] (0,0)--(8mm,0);} $y= R^{1+\phi(R)}$ \\
    \tikz{\fill[gray!20] (0,0) rectangle (0.8, 0.25); \node[anchor=west] at (0.8, 0.125) {\tiny $\mathcal{D}(R_{n}, k_n t,t^{1+\theta_{n-1}})\setminus $};} \\
    \tikz{\filldraw[white] (0, 0) rectangle (0.8, 0.25); \node[anchor=west] at (1.4, 0.125) {\tiny  $\mathcal{D}(R_n^\prime,k_n t,t^{1+\frac{\log k_{n}}{\log R_n^\prime}})$};}
};
\end{tikzpicture}

\caption{Figure of $R^{1+\phi(R)}$ for $R \in [R_{n+2}, R_n]$}
\label{low dim Rphi 2}
\end{figure}
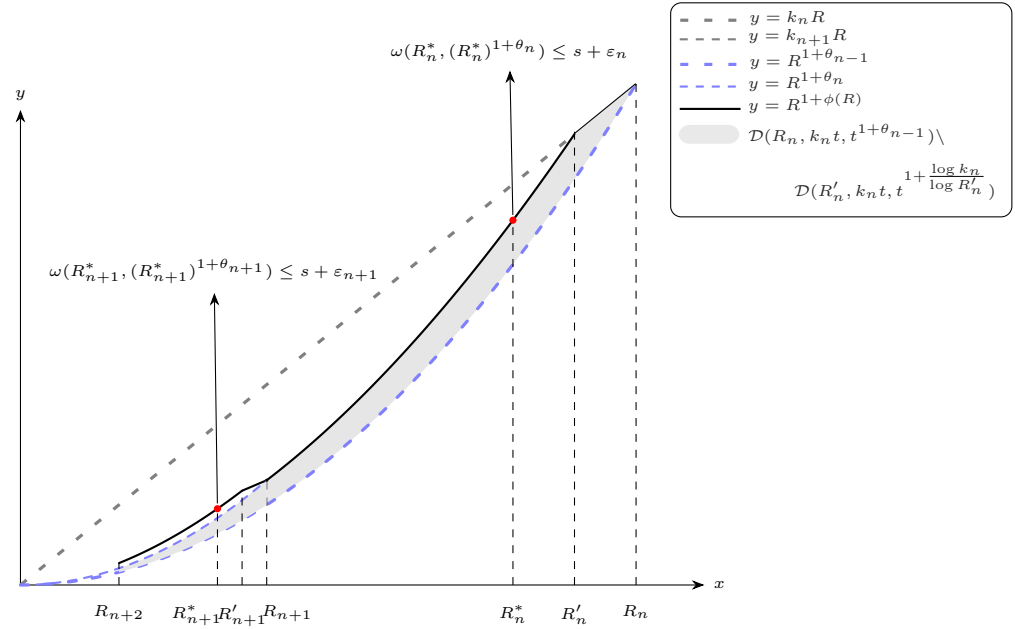

We can easily verify that $\phi(R)$ is a dimension function, and the desired conclusion holds by the definition of $\phi(R)$ and Claim \ref{cm3}.
\end{proof}

\subsection{Proof of Theorem 4: Generalised lower dimensions on the popcorn graphs}

In this part, we discuss the generalised lower-type dimensions of popcorn graphs, and particularly demonstrate the interpolation problem on the generalised lower-type dimension analogues of popcorn graphs. For general discuss on the dimension theory of popcorn graphs and their variations, see \cite{BC2023, C2022, CFY2022, DWW2023} and reference therein.

We first recall the definition of popcorn function and popcorn graphs. 

Given $0 < t < \infty$. Let $f: [0,1] \to [0,1)$ be the {\it popcorn function} where for any $x \in [0,1]$ 
\[
f_t(x) = 
\begin{cases}
\frac{1}{q^t} & x=\frac{p}{q} \text{ where } 1 \leq p \leq q, \, \gcd(p,q) = 1 \\
0 & \text{otherwise}
\end{cases}
\]
We denote
\[
S_{t} = \left\{ (x, f_t(x)) ~:~ x \in (0,1) \cap \mathbb{Q} \right\} \subseteq \RR^2.
\]
be the {\it popcorn graph}.

It follows from \cite{C2022, CFY2022, DWW2023} that
\[
\dim_{\rm B} S_t = 
\begin{cases}
\frac{4}{2+t} & 0 < t < 2; \\
1 & 2 \leq t < \infty,
\end{cases}
\]
Note that $S_t$ contains isolated points, thus for any dimension function $\phi$, it gives $\lad^\phi S_t = 0$. Besides, notice that 
\[
([0,1] \backslash \mathbb{Q}) \times \{0\} \subseteq S_t,
\]
thus by \cite[Section 3.4.2]{JF2020}, for all dimension functions $\phi$, $\dim_{\rm ML}^\phi S_t = 1$, which gives for any $0 < t < 2$,
\[
\lad^\phi S_t < \dim_{\rm ML}^\phi S_t < \dim_{\rm B} S_t.
\]
This also shows that the generalised lower-type dimensions fail to fully interpolate from $\lim_{\theta \to 0} \dim_{\rm L}^\theta F$ to $\dim_{\rm B} F$ for general sets.

\section*{Acknowledgements.} 

The authors would like to thank Dr. Amlan Banaji and Dr. Alex Rutar for their profound friendship, helpful discussions and constant encouragement. H. P. Chen would like to thank Dr. Amlan Banaji for hosting and providing an excellent research atmosphere during his visit at the University of Jyv\"askyl\"a. H. P. Chen was financially supported by NSFC 12401107 and JYU Visiting Fellow Programme 2026. W. Wang was supported by NSFC 12061086.

\newpage

\section*{Appendix: A new proof on recovering lower Assouad dimension}

In this appendix, we give an alternative proof on the recovery of lower Assouad dimension by dimension functions. This alternative proof starts from an equivalent definition of lower Assouad dimension, and then prove the recovery of lower Assouad dimension by giving the direct construction of the corresponding dimension function. We first study the equivalent definition of lower Assouad dimension by dimension function as follows.

\begin{proposition}\label{equiv def}
Let $\Phi(R)$ be a monotonic decreasing dimension function satisfying that $R^{\Phi(R)} \to 0$ as $R \to 0$, then 
\begin{align}\label{equi.def}
\lad F = \sup\{s \geq 0 \mathdot & \exists \, c > 0, \text{ such that }  \forall \,  0 <  R^{1+ \Phi(R)} \leq r < R < 1, \\
& \inf_{x \in F} N_r(B(x,R) \cap F)  \geq c \cdot  (R/r)^s \}. \notag
\end{align}
\end{proposition}

\begin{proof}

We denote RHS to be the right hand side of (\ref{equi.def}). It directly follows from  definition that $\lad F \leq \text{RHS}$.
For the converse, it suffices to prove for any $s < {\rm RHS}$ and sufficiently small $\varepsilon > 0$, and for any $R > 0$ and $0 <  r < R^{1+\Phi(R)}$, 
\[
M_r(B(x,R) \cap F) \geq \left( \frac{R}{r} \right)^{s - \varepsilon}.
\]
Given $s$ and $\varepsilon$, for any sufficiently small $R >0$ and $R^{1+ \Phi(R)} \leq r < R$, it gives
\[
M_r(B(x,R) \cap F) \geq R^{-\Phi(R) (s - \eps)}
\]
and 
$
R^{\Phi(R)} < \varepsilon.
$
Given any $r < R^{1+ \Phi(R)}$, we define $R_0 = R$, and for any $n \geq 1$, 
\[
R_{n} = R_{n-1}^{1 + \Phi(R_{n-1})}.
\]
Hence there exist $n(r)$ such that
\[
R_{n(r) + 1} < r \leq R_{n(r)}.
\]
Then by Lemma \ref{lower bound estimate} and \ref{lem4}, 
if $\frac{R_{n(r)}}{2} < r \leq R_n$, then
\begin{align*}
M_r(B(x,R) \cap F) & \geq M_{R_1}(B(x,R - R_1) \cap F) \cdot \inf_{y_1 \in F} M_{R_2}(B(y_1,R_1 - R_2) \cap F) \cdot \dots \\
& \qquad \qquad \cdot \inf_{y_{n(r)-2} \in F}M_{r}(B(y_{n(r)-2},R_{n(r)-1} - R_{n(r)}) \cap F) \\
& \geq \left(\frac{R}{R_1}\right)^{s-\varepsilon} \cdot \left(\frac{R_1}{R_2}\right)^{s-\varepsilon} \cdot \dots  \cdot \left(\frac{R_{n(r)-1}}{R_{n(r)}}\right)^{s-\varepsilon} \geq \left(\frac{R}{2r}\right)^{s-\varepsilon}.
\end{align*}
Otherwise, for $R_{n(r)+1} < r \leq \frac{R_{n(r)}}{2}$, then by (\ref{equiv def}) and Lemma \ref{lem3}, there exist a constant $c_0 > 0$ such that for any $y \in F$,
\begin{equation}
M_{r}(B(y,R_{n(r)} - r) \cap F) \geq c_0 \left( \frac{R_{n(r)}}{r} \right)^{s-\varepsilon},
\end{equation}
hence
\begin{align*}
M_r(B(x,R) \cap F) & \geq M_{R_1}(B(x,R - R_1) \cap F) \cdot \inf_{y_1 \in F} M_{R_2}(B(y_1 ,R_1 - R_2) \cap F) \cdot \cdots \\
& \qquad \qquad \cdot \inf_{y_{n(r)-1} \in F}M_{r}(B(y_{n(r)-1},R_{n(r)} - r) \cap F) \\
& \geq c_0 \cdot \left(\frac{R}{R_1}\right)^{s-\varepsilon} \cdot \dots  \cdot \left(\frac{R_{n(r)-1}}{R_{n(r)}}\right)^{s-\varepsilon}\cdot \left( \frac{R_{n(r)}}{r} \right)^{s-\varepsilon} \geq \left(\frac{R}{r}\right)^{s-\varepsilon}
\end{align*}
which gives the desired result.
\end{proof}

Based on Proposition \ref{equiv def}, we have a refinement of the lower Assouad dimension under dimension functions. This is a parallel result of \cite[Theorem 2.9]{BRT2023}.

\begin{proposition}\label{equiv def 2}
Let $F$ be a n.b.d. space, and $g: (0,1) \to (0,1)$ be a continuous function satisfying that 
\[
\lim_{R \to 0} g(R)/R = 0.
\]
Then there exist a dimension function $\psi$ with $R^{1+\psi(R)} \geq g(R)$ for all $R \in (0,1)$ such that 
\[
\dim_{\rm L} F  = \dim_{\rm L}^\psi F = \underline{\dim}_{\rm L}^\psi F.
\]
\end{proposition}

\begin{proof}
Let $\phi(R) = (\log g(R)/\log R) - 1$, and it follows from \cite[Proposition 2.6]{BRT2023} that there exist the unique maximal dimension function $\psi_0(R) : (0,1) \to (0,1)$ as
\[
\psi_0(R) = \inf_{R^\prime \in (R,1)} \frac{\inf_{0 < r \leq R^\prime}\phi(r)\log(1/r)}{\log(1/R^\prime)}.
\] 
It follows that $\psi_0(R) \leq \phi(R)$ and $\psi_0(R)$ is monotonic decreasing as $R \to 0$. Then by Proposition \ref{equiv def}, it also gives $\dim_{\rm L} F = \dim_{\rm L}^{\psi_0} F$. Moreover, since $F$ is doubling, then there exist a sequence $\{x_n, R_n, r_n \}_{n=1}^\infty$ satisfying that $r_n \geq R_n^{1+\psi_0(R_n)}$, $r_n /R_n \to 0$ as $n \to \infty$, and 
\begin{equation}\label{equiv def 2 eqn1}
\dim_{\rm L} F = \lim_{n \to \infty} \frac{\log N_{r_n}(B(x_n, R_n) \cap F)}{\log (R_n/r_n)}.
\end{equation}

We turn to introduce the construction of dimension function $\psi$. 
For any $n$, we denote $\theta_n$ by the solution satisfying $r_n = R_n^{1+\theta_n}$. Notice that $r_n \geq R_n^{1+\psi_0(R_n)}$ gives $\theta_n \leq \psi_0(R_n)$. Besides, $r_n/R_n \to 0$ as $n \to \infty$ implies that $\theta_n \log (1/R_n)$ goes to infinity as $n \to \infty$. Since $\psi_0(R)$ decreases as $R \to 0$, thus by passing a subsequence, we assume that $\theta_n$ monotonically converges to some $\theta \in [0,\infty)$. If $\theta > 0$, then let $\psi(R) = \theta$ for all $R \in (0,1)$, then it follows from lemma \ref{lem4} and (\ref{equiv def 2 eqn1}) that $\dim_{\rm L} F  = \dim_{\rm L}^\psi F$ and $\psi(R) \leq \theta_n \leq \psi_0(R)$ for all $R \in (0,1)$.

Otherwise, we may assume that $\theta_n$ decreases to $0$, and we introduce the construction of $\psi$ on $(0,1) $ by induction. Let $\psi(R) = \theta_1$ for $R \in [R_1,1)$, then we define $R_1^\prime < R_1$ satisfying 
\[
\theta_1 \frac{\log R_1}{\log R_1^\prime} = \theta_2,
\]
and we define $\psi(R) = \theta_1 \cdot (\log R_1/\log R)$ on $[R_1^\prime, R_1]$, and $\psi(R) = \theta_2$ on $[R_2, R_1^\prime]$. Thus, by induction, for any $n \geq 1$, we define $R_n^\prime < R_n$ be the solution satisying 
\[
\theta_n \frac{\log R_n}{\log R_n^\prime} = \theta_{n+1},
\]
and for any $R \in (0,1)$, we define 
\begin{equation}
\psi(R) = 
\begin{cases}
\theta_n \cdot (\log R_n/\log R) & R \in [R_n^\prime, R_n]; \\
\theta_{n+1} & R \in [R_{n+1}, R_n^\prime].
\end{cases}
\end{equation}
This directly gives that $\psi(R)$ is monotonic decreasing and $\psi(R)\log (1/R) \to \infty$ as $R\to 0$,and $\psi(R)$ is a dimension function. Moreover, for any $R \in (0,1)$ and any $n \geq $, by the construction of $\psi(R)$, $\theta_{n} \leq \psi_0(R_{n})$ and $\psi_0(R)$ monotonic decreasing as $R \to 0$, then it gives $\psi(R) \leq \psi_0(R)$.

It remains to check 
$
\dim_{\rm L} F  = \dim_{\rm L}^\psi F = \underline{\dim}_{\rm L}^\psi F.
$
Note that 
\begin{align*}
\dim_{\rm L} F  \leq \underline{\dim}_{\rm L}^\psi F  & \leq \dim_{\rm L}^\psi F \\
& = \liminf_{R \to 0} \frac{\log \inf_{x \in F} N_{R^{1+\psi(R)}}(B(x,R) \cap F)}{-\psi(R) \log R} \\
& \leq \liminf_{n \to \infty} \frac{\log \inf_{x \in F} N_{R_n^{1+\psi(R_n)}}(B(x,R_n) \cap F)}{-\psi(R_n) \log R_n} \\
& =  \lim_{n \to \infty} \frac{\log N_{r_n}(B(x_n, R_n) \cap F)}{\log (R_n/r_n)} \quad (\text{ by (\ref{equiv def 2 eqn1})}) \\
& = \dim_{\rm L} F .
\end{align*}
This yields the desired argument.
\end{proof}

We now investigate the recovery of lower Assouad dimension in terms of dimension functions. By (\ref{omega x y}) and (\ref{regions}), the following claim gives an formulation of Proposition \ref{equiv def 2}.


\begin{claim}\label{cm1}
  Let $\phi(x)$ be a monotonic decreasing dimension function satisfying that $x^{\phi(x)} \to 0$ as $x \to 0$. Let $\dim_{\mathrm L} F=s$, then for any $R>0$, $k>0$, and $\varepsilon>0$, there exists $(x,y)\in\mathcal{D}(R,kt,t^{1+\phi(t)})$, such that $\omega(x, y)\leq s+\varepsilon$.
\end{claim}
Besides, directly follows from the the definition of $\dim_{\rm L}^\phi F$, the following technical claim is instrumental in recovering lower Assouad dimension by using dimension functions.
\begin{claim}\label{cm2}
Let $s  = \dim_{\mathrm L}F $, for any sequence $\{\varepsilon_i\}_{i=1}^{\infty}$ converging to $0$, if there exists a  dimension function $\phi(x)$ and sequence $\{R_i\}_{i=1}^{\infty}$ decreasing to $0$, such that $\omega(R_i, R_i^{1+\phi(R_i)})\leq s+\varepsilon_i$, then $\dim_{\mathrm L}^{\phi} F = \dim_{\mathrm L} F $.\end{claim}

With the help of Claims \ref{cm1} and \ref{cm2}, we state the proof of recovering the lower Assouad dimension by dimension function as follows.
\begin{proposition}\label{THM_interpolation_1}
Let $s = \dim_{\mathrm L} F$, then there exist a dimension function $\phi(x)$ such that
$
	 \dim_{\mathrm L} F = \dim_{\mathrm L}^\phi F=s.
$
\end{proposition}

\begin{proof}
Given  sequence $\{\varepsilon_i\}_{i=1}^{\infty}$ converging to $0$, we first apply Claim $\ref{cm1}$ to construct sequences $\{\theta_i\}_{i=1}^\infty$, $\{(R_i^*,{R_i^*}^{1+\theta_i})\}_{i=1}^\infty$, and the corresponding dimension function $\phi(x)$ satisfying $\omega(R_i^*,{R_i^*}^{1+\theta_i})\leq s+\varepsilon_i$ and $\phi(R_i^*)=\theta_i$. We then finish the proof by Claim $\ref{cm2}$.
	
We first give some notations. For any $0 < R < \infty$, and for any non-negative function $f(t)$, we define 
\begin{align*}
	\mathcal{L}(R,f(t)) & = \bigbrackets{(t,y) \mathdot 0< t\leq R, \,  y = f(t) }.
\end{align*}
Given $0<k<1$ and $0<t\leq R$ , we denote $\theta(k,t) = \frac{\log k }{\log t}$, and it directly follows that
\begin{equation}
	(t,kt)\in\mathcal{L}(R,kt)\cap\mathcal{L}(R,t^{1+\theta(k,t)}).
\end{equation}
In other words, $\mathcal{L}(R,t^{1+\theta(k,t)})$ passes through the point $(t,kt)$.

We turn to construct sequences $\{\theta_i\}_{i=1}^\infty$ and $\{(R_i^*,{R_i^*}^{1+\theta_i})\}_{i=1}^\infty$ by induction, and then the corresponding dimension function $\phi(x)$ by $\{\theta_i\}_{i=1}^\infty$.
Given $R_1=\frac{1}{2}$,  $k_1=\frac{1}{2}$, we define $\theta_1= \frac{\log k_1}{\log R_1}$. It follows from Claim \ref{cm1} that there exists $(R_1^*,y_1)\in \mathcal{D}(R_1,k_1x,x^{1+\theta_0}) $ such that 
	\[\omega(R_1^*,y_1)\leq s+\varepsilon_1.\]
	Notice that 
	\[\mathcal{D}(R_1,k_1 t, t^{1+\theta_1})=\bigcup_{0 < x \leq R_1}\mathcal{L}(x,t^{1+\frac{\log k_1}{\log x}}),\]
where $\frac{\log k_0}{\log x}$ decreases monotonically from $\theta_0$ to $0$ as $x$ goes from $R_0$ to $0$, so there exist $0< \theta_1 <\theta_0$ and $R_1^*$ such that $y_1={(R_1^*)}^{1+\theta_1}$. We also emphasize that 
\[
(R_1^*,{R_1^*}^{1+\theta_1})\in \mathcal{L}(R_1^\prime, t^{1+\theta_1}).
\]
Let $R_2=\frac{1}{2} R_1^*$,  $k_1=R_1^{\theta_1}$, we define
\[
\phi(R) = 
\left\{
\begin{array}{ll}
	\frac{\log k_1}{\log R}& R_1^\prime<R\leq R_1;\\
	\theta_1 & R_2<R\leq R_1^\prime.
\end{array}
\right.
\]

Recursively, if  $R_n^*$, $R_n^\prime$, $\theta_n$ have been defined, then let $R_{n+1}=\frac{1}{2}R_n^*$, $k_{n+1}=R_n^{1+\theta_n}$, It follows from Claim \ref{cm1} that there exists $R_{n+1}^* < R_{n+1}^\prime$ such that
 \[
(R_{n+1}^*, {(R_{n+1}^*)}^{1+\theta_{n+1}})\in \mathcal{L}(R_{n+1}^\prime,t^{1+\theta_{n+1}})\subset\mathcal{D}({R_{n+1},k_{n+1}t,t^{1+\theta_{n+1}}})
\]
and
\[\omega(R_{n+1}^*, {R_{n+1}^*}^{1+\theta_{n+1}})\leq s+\varepsilon_{n+1}\]
where $\theta_{n+1}=\frac{\log k_{n+1}}{\log R_{n+1}^\prime}$. 
As a consequence, for any $0 < R <R_1$, we define
\[
\phi(R) = 
\left\{
\begin{array}{ll}
	\frac{\log k_n}{\log R}& R_{n+1}^\prime<R\leq R_n  \\
	\theta_{n+1} & R_{n+2}<R\leq R_{n+1}^\prime
\end{array}.
\right.
\]
Details of $R^{1+\phi(R)}$ see figure \ref{low dim Rphi}.
\begin{figure}[htbp]
\centering
\begin{tikzpicture}[font=\tiny, >=Stealth, scale=0.45]

  \pgfmathsetmacro{\Rn}{16}        
  \pgfmathsetmacro{\m}{0.82}         
  \pgfmathsetmacro{\pOut}{3}         
  \pgfmathsetmacro{\pDas}{1.5}       

  \pgfmathsetmacro{\xA}{0.4*\Rn}     
  \pgfmathsetmacro{\xB}{0.8*\Rn}     
  \pgfmathsetmacro{\xC}{0.9*\Rn}     
  \pgfmathsetmacro{\yRn}{\m*\Rn}
  \pgfmathsetmacro{\yA}{\m*\xA}
  \pgfmathsetmacro{\yB}{\m*\xB}
  \pgfmathsetmacro{\yC}{\m*\xC}

  \pgfmathsetmacro{\yApp}{\m*\xA*pow(\xA/\xC, \pDas-1)}
  \pgfmathsetmacro{\mp}{\yApp/\xA}         
  \pgfmathsetmacro{\xAp}{0.4*\xA}    
  \pgfmathsetmacro{\xBp}{0.8*\xA}    
  \pgfmathsetmacro{\xCp}{0.9*\xA}    
  \pgfmathsetmacro{\yAp}{\mp*\xAp}
  \pgfmathsetmacro{\yBp}{\mp*\xBp}
  \pgfmathsetmacro{\yCp}{\mp*\xCp}
  \pgfmathsetmacro{\pOutt}{ln(\yApp)/ln(\xA)}

  \draw[-Stealth] (0,0) -- (17,0) node[right] {$x$};
  \draw[-Stealth] (0,0) -- (0,12) node[above] {$y$};

  \draw[gray,very thick, loosely dashed] (0,0) -- (\xC,\yC);         
  \draw[gray,dashed] (0,0) -- (\xA,\yApp);         
  \draw[thick]      (\xC,\yC) -- (\Rn,\yRn);    
  \draw[thick]      (\xCp,\yCp) -- (\xA,\yApp);

  \draw[blue!50, very thick, loosely dashed, line cap=round]
    plot[domain=0:\Rn, samples=240]
      (\x, { \m*\x*pow(\x/\Rn, \pOut-1) });
  \draw[blue!50, dashed, line cap=round]
    plot[domain=0:\xA, samples=240]
      (\x, { \mp*\x*pow(\x/\xA, \pOutt) });

  \draw[thick, line cap=round]
    plot[domain=\xA:\xC, samples=240]
      (\x, { \m*\x*pow(\x/\xC, \pDas-1) });

  \draw[thick, line cap=round]
    plot[domain=\xAp:\xCp, samples=240]
      (\x, { \mp*\x*pow(\x/\xCp, \pDas-1) });

  \draw[dashed] (\xA,0) -- (\xA,\yApp) node[below=4pt] at (\xA + 0.8,0) {$R_{n+1}$};
  \draw[dashed] (\xB,0) -- (\xB,\yB) node[below=4pt] at (\xB,0) {$R_n^{*}$};
  \draw[dashed] (\xC,0) -- (\xC,{\yC}) node[below=4pt] at (\xC,0) {$R_n'$};
  \draw[dashed] (\Rn,0) -- (\Rn,{\yRn}) node[below=4pt] at (\Rn,0) {$R_n$};

  \draw[dashed] (\xAp,0) -- (\xAp,\yAp-0.4) node[below=4pt] at (\xAp ,0) {$R_{n+2}$};
  \draw[dashed] (\xBp,0) -- (\xBp,\yBp) node[below=4pt] at (\xBp -0.8, 0) {$R_{n+1}^{*}$};
  \draw[dashed] (\xCp,0) -- (\xCp,{\yCp}) node[below=4pt] at (\xCp , 0) {$R_{n+1}'$};

  \pgfmathsetmacro{\yDot}{  \m*\xB *pow(\xB /\xC , \pDas-1) }
  \pgfmathsetmacro{\yDotp}{ \mp*\xBp*pow(\xBp/\xCp, \pDas-1) }
  \fill  [red] (\xB ,\yDot ) circle (3pt);
  \fill [red] (\xBp,\yDotp) circle (3pt);

  \coordinate (P1) at (\xB ,\yDot );
  \node (txt) at (\xB -0.1 ,\yDot +4 ) {$\omega(R_{n}^*, {(R_{n}^*)}^{1+\theta_{n}})\leq s+\varepsilon_{n}$};
  \draw[->, shorten <=2pt] (P1) -- (txt);

  \coordinate (P2) at (\xBp,\yDotp);
  \node (txt) at (\xBp -0.1,\yDotp + 6) {$\omega(R_{n+1}^*, {(R_{n+1}^*)}^{1+\theta_{n+1}})\leq s+\varepsilon_{n+1}$};
  \draw[->, shorten <=2pt] (P2) -- (txt);

\node[
  draw, fill opacity=0.85, rounded corners,
  inner sep=6pt, row sep=6pt, column sep=6pt,
  anchor=north east,  align=left
] at (24,15) {
  \tikz{\draw[gray, very thick, loosely dashed, line width=0.9pt]  (0,0) -- (8mm,0);}  $y=k_n x$ \\
  \tikz{\draw[gray, dashed, line width=0.9pt]  (0,0) -- (8mm,0);}  $y=k_{n+1} x$ \\
  \tikz{\draw[blue!50, very thick, loosely dashed, line width=0.9pt] (0,0) -- (8mm,0);}  $y=x^{1+\theta_{n-1}}$ \\
  \tikz{\draw[blue!50, dashed, line width=0.9pt] (0,0) -- (8mm,0);}  $y=x^{1+\theta_{n}}$  \\ \tikz{\draw[black, thick, line width=0.9pt] (0,0) -- (8mm,0);}  $y=R^{1+\phi(R)}$ 
};

\end{tikzpicture}
    \caption{Figure of $R^{1+\phi(R)}$ for $R \in [R_{n+2}, R_n]$.}
    \label{low dim Rphi}
\end{figure}
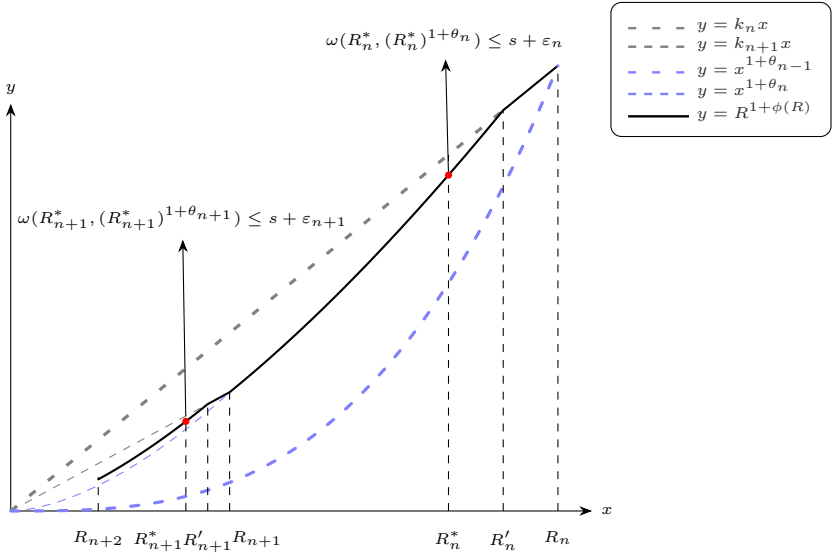

We can easily verify that $\phi(R)$ is a dimension function,
from the definition of $\phi(R)$ and Claim $\ref{cm2}$, we obtain the desired conclusion. 

\end{proof}


\begin{thebibliography}{10}

\bibitem{B2023} Banaji A. {\it Generalised Intermediate dimensions.} Monatsh Math., 202 (2023), 465-506.

\bibitem{BC2023} Banaji A., Chen H.-P., {\it Dimensions of popcorn-like pyramid sets}, J. Fractal Geom. no. 1, 10 (2023), 151-168.

\bibitem{BR2022} Banaji A., Rutar A., {\it Attainable forms of intermediate dimensions.} Ann. Fenn. Math., no.2, 47 (2022), 939-960.

\bibitem{BRT2023} Banaji A., Rutar A., Troscheit S., {\it Interpolating with Generalized Assouad dimensions.} Geom. Anal., 35 (2025), 270.

\bibitem{C2022} Chen H.-P., {\it Dimensions and spectra of the t-popcorn graphs}, J. Math. Anal. Appl., no.1, 510 (2022), 126013.

\bibitem{CDW2017} Chen. H.-P., Du Y.-L., Wei C., {\it Quasi-lower dimension and quasi-Lipschitz mapping.} Fractals. no.3, 25 (2017), 1750034.

\bibitem{CFY2022} Chen H.-P., Fraser J. M., Yu H., {\it Dimensions of the popcorn graph}, Proc. Amer. Math. Soc, no. 11, 150 (2022), 4729-4742.

\bibitem{CWC2020} Chen. H.-P., Wu M., Chang Y.-Y. {\it Lower Assouad-type dimensions of uniformly perfect sets in doubling metric spaces.} Fractals., no. 2. 28 (2020), 2050039.

\bibitem{DWW2023} Du Y.-L., Wei C., Wen S. {\it dimensions of popcorn subsets}, J. Math. Anal. Appl., no. 2, 524 (2023), 127088.

\bibitem{KF2014} Falconer K. J., Fractal Geometry: Mathematical Foundations and Applications, Wiley \& sons, New York (2014).

\bibitem{JF2020} Fraser J. M., Assouad Dimension and Fractal Geometry. Cambridge University Press, Cambridge (2020).

\bibitem{FY2018} Fraser J. M., Yu H., {\it New dimension spectra: Finer information on scaling and homogeneity}, Adv. Math., 329 (2018), 273-328.

\bibitem{FHKY2019} Fraser J. M., Howroyd J. D., K\"aenm\"aki A., Yu H., {\it On the Hausdorff dimension of microsets}, Proc. Amer. Math. Soc., no. 11, 147 (2019), 4921-4936.

\bibitem{GHM2021} Garcia I., Hare K., Mendivil F., {\it Intermediate Assouad-like dimensions.} J. Fractal Geom. 8 (2021), 201-245.

\bibitem{HT2019} Hare K., Troscheit S., {\it Lower Assouad Dimension of Measures and Regularity.} Math. Proc. Camb. Philos. Soc. no.2, 170 (2021), 379-415.

\bibitem{L1967} Larman D.G., {\it A New Theory of Dimension}, P. Lond. Math. Soc., no.1, s3-17 (1967),  178-192.

\bibitem{R2024} Rutar A., {\it Attainable forms of Assouad spectra.}, Indiana Univ. Math. J., 73 (2024), 1331-1356.
\end{thebibliography}
\end{document}